\documentclass[11pt,a4paper,
fleqn,reqno]{amsart}                 
\usepackage[T1]{fontenc}                   
\usepackage[utf8]{inputenc}                 
\usepackage[english]{babel}    
 \usepackage{csquotes}
\usepackage{graphicx}                      %
\usepackage[font=small]{quoting}    
\usepackage{float}
\usepackage{caption}

\usepackage[top=1.4in,bottom=1.8in,left=1.2in,right=1.2in]{geometry}                  
\usepackage{verbatim}
\usepackage{color}
\usepackage{xcolor}
\usepackage{picture}
\usepackage{amsmath,amsthm,amsfonts,amssymb}
\usepackage{comment}
\usepackage{hyperref}
\hypersetup{colorlinks,linkcolor={black}, citecolor={black},urlcolor={black}}  
\usepackage{enumitem}
\usepackage{mathtools}
\usepackage[url=false,doi=false,isbn=false, giveninits=true,
style=
numeric,
maxbibnames=99]{biblatex}
\newtheorem{thm}{Theorem}[section]

\newtheorem{prop}[thm]{Proposition} 
 
\newtheorem{defn}[thm]{Definition}

\theoremstyle{definition} 
\newtheorem{rem}[thm]{Remark} 

\theoremstyle{definition} 
\newtheorem{conjecture}[thm]{Conjecture}

\def\O{\Omega}

\def\S{\Sigma} 
\def\n{\nabla}

\def\p{\partial}

\def\a{\alpha}

\def\n{\nabla}

\def\o{\omega}
\def\O{\Omega}
\def\p{\partial}

\def\a{\alpha}

\def\g{\gamma}
\def\d{\delta}
\def\k{\kappa}
\def\l{\lambda}
\def\s{\sigma}

\def\ov{\overline}

\def\n{\nabla}
\def\<{\langle}
\def\>{\rangle}

\def\n{\nabla}

\def\RR{\mathbb{R}}

\def\SS{\mathbb{S}}

\def\o{\omega}
\def\O{\Omega}
\def\p{\partial}
\def\F{\mathcal{F}}

\def\a{\alpha}

\def\g{\gamma}
\def\G{\mathcal{G}}
\def\d{\delta}
\def\l{\lambda}
\def\K{\mathcal{K}}
\def\s{\sigma}

\def\ov{\overline}

\def\wh{\widehat}

\def\R{\mathbb{R}}

\def\E{\mathcal{E}}
\def\C{\mathcal{C}}

\def\L{\mathcal{L}}

\def\ol{\overline}
\def\wt{\widetilde}
\def\DD{\mathcal{D}}

{\left\{\begin{array}{@{}l@{}}}{\end{array}\right.}
\patchcmd{\abstract}{\scshape\abstractname}{\textbf{\abstractname}}{}{}
\makeatletter 
\def\@makefnmark{} 
\makeatother

\numberwithin{equation}{section}
\numberwithin{exa}{section}

\begin{document}
\title[Capillary Gauss curvature flow]{The capillary Gauss curvature flow}
 
\author[X. Mei]{Xinqun Mei}
\address[X. Mei]{Key Laboratory of Pure and Applied Mathematics, School of Mathematical Sciences, Peking University,  Beijing, 100871, P.R. China}

\email{\href{qunmath@pku.edu.cn}{qunmath@pku.edu.cn}}

\author[G. Wang]{Guofang Wang}
\address[G. Wang]{Mathematisches Institut, Albert-Ludwigs-Universit\"{a}t Freiburg, Freiburg im Breisgau, 79104, Germany}
\email{\href{guofang.wang@math.uni-freiburg.de}{guofang.wang@math.uni-freiburg.de}}

\author[L. Weng]{Liangjun Weng}
\address[L. Weng]{Centro di Ricerca Matematica Ennio De Giorgi, Scuola Normale Superiore, Pisa, 56126, Italy \& Dipartimento di Matematica, Universit\`a di Pisa, Pisa, 56127, Italy}
\email{\href{mailto:liangjun.weng@sns.it}  {liangjun.weng@sns.it}}



\begin{abstract}
In this article, we first introduce a Gauss curvature type flow for capillary hypersurfaces, which we call {\it capillary Gauss curvature flow}. We then show that the flow will shrink to a point in finite time. 
This is a capillary counterpart (or Robin boundary counterpart) of Firey's problem studied by Firey \cite{Firey} and Tso \cite{Tso85}. Finally, we prove that its normalized flow converges to a soliton. This is a capillary counterpart of the result of  Guan and Ni in \cite{GN}. The classification of solitons remains an open conjecture.
\end{abstract}

\maketitle


 \section{Introduction}\label{sec1}
Studying the evolution of closed hypersurfaces under geometric curvature flows is an important topic in geometric analysis. It lies in the intersection of geometry, partial differential equations, and applied mathematics. One of the well-known models is the Gauss curvature flow for closed hypersurfaces, which was proposed by Firey \cite{Firey}. This model is used to describe the changing shape of a tumbling stone subjected to collisions from all directions with uniform frequency. It can be formulated as a family of smooth, strictly convex, closed hypersurfaces $\{M_{t}=X(\SS^n,t)\}_{t\geq 0}\subset \mathbb{R}^{n+1}$ 
satisfying 
\begin{eqnarray}\label{alpha-Gauss curvature flow}
   \p_t X(x,t)=-K (x,t)\nu(x,t), ~~~~x\in \SS^n, ~t\geq 0,
\end{eqnarray}
where $K(x,t)$ is the Gauss curvature and $\nu(x, t)$ is the unit outward normal at the point $X(x,t)\in M_{t}$. 

Firey \cite{Firey} investigated the asymptotic behavior of the solution to flow \eqref{alpha-Gauss curvature flow} under the assumption of the existence, regularity, and uniqueness of the solution. He proved that if the initial hypersurface $M_{0}\subset \mathbb{R}^{3}$ is centrally symmetric, then the evolving hypersurface will gradually become more and more round along flow \eqref{alpha-Gauss curvature flow} (in the sense that its normalized flow converges to a round sphere).  He conjectured that the above conclusion also holds without the symmetric assumption. Later, Tso \cite{Tso85} proved that flow \eqref{alpha-Gauss curvature flow} shrinks to a point in finite time for all dimensions. When $n=2$, Andrews \cite{And99} solved this conjecture completely, i.e., the flow converges to a round point. Recently, Guan and Ni \cite{GN} proved that its normalized flow converges to a Gauss soliton, a solution to the following equation 
\begin{eqnarray}\label{GCF similar solution}
    K(x)=\<X(x), \nu(x)\> ,\qquad x\in \SS^n.
\end{eqnarray}
Hence, in order to solve  Firey's conjecture for the general $n$, we need only to show that all solutions to \eqref{GCF similar solution} are spheres, which was confirmed very recently in a 
breakthrough work of Choi and Daskalopoulos \cite{CD}.

In this paper, we first introduce a capillary counterpart of the Gauss curvature flow. A properly embedded, smooth compact hypersurface $\S$ in $\ov{\R^{n+1}_+}$ with boundary $\partial \S \subset \p \R^{n+1}_+$ is called a \textit{capillary hypersurface}  if it intersects $\p\R^{n+1}_+$ at a constant contact angle $\theta \in (0,\pi).$ In particular, if $\theta=\frac \pi 2$, we call it \textit{free boundary hypersurface}. Denote $\wh{\S}$ be the bounded closed domain in $\ol{\RR^{n+1}_+}$, which is enclosed by $\S$ and $\p\RR^{n+1}_+$. The boundary of $\widehat{\S}$  consists of two parts: one is $\Sigma\subset\RR^{n+1}_+$ and  the other, which will be denoted by $\widehat{\p \Sigma}$,  lies on  $\p \RR^{n+1}_+$.  A simple example of a capillary hypersurface is the unit spherical cap $$\C_{\theta}\coloneqq  \left\{\xi\in \ov{\mathbb{R}^{n+1}_+} \mid ~|\xi-\cos\theta e|= 1 \right\},$$ where $-e\coloneqq E_{n+1}$ is the $(n+1)$-th unit vector of $\ol{\RR^{n+1}_+}$. 
A capillary Gauss curvature flow is a family of strictly convex capillary hypersurfaces $\S_t$ in $\ov{\R^{n+1}_+}$, given by embeddings $X\colon  M\times[0,T')\to \ol{\RR^{n+1}_+}$ and satisfying
\begin{eqnarray}\label{c-Gauss curvature flow}
   \p_t X(x,t)=-K (x,t)\tilde\nu(x,t), ~~~~~(x,t)\in M\times [0, T'),
\end{eqnarray}with a contact angle boundary condition \begin{eqnarray}\label{cap bdry condition}
    \<\nu(x,t), e\>=\cos(\pi-\theta) ,~~~~ ~ (x,t)\in \p M\times[0, T'),
\end{eqnarray}
where \begin{eqnarray}\label{cap Gauss map}
\tilde \nu\coloneqq \nu+\cos\theta e,
 \end{eqnarray}  is the capillary Gauss map of $\S_t$. See Section \ref{sec2.1} or \cite[Section 2]{MWWX} for more discussion on $\tilde \nu$. 
 
 Analogous to the classical Gauss curvature flow studied by Firey, which models the isotropic wear of convex closed hypersurfaces, the capillary version \eqref{c-Gauss curvature flow} describes the evolution of convex hypersurfaces with boundary under both curvature effects and boundary interaction. In this setting, the hypersurface evolves with normal speed equal to its Gauss curvature, while maintaining a fixed contact angle with the supporting hyperplane $\partial\mathbb{R}^{n+1}_+$. This flow provides a geometric model for the relaxation of sessile droplets on flat substrates, where motion is driven by surface tension (via Gauss curvature), and the prescribed contact angle reflects adhesion at the boundary. It captures key features of capillarity-driven spreading and the curvature-induced evolution of interfaces constrained by rigid supports. 
 
 Notice that \eqref{cap bdry condition} implies $\<\tilde \nu(x,t), e\>=0$ on $\p M\times[0, T')$. The flow \eqref{c-Gauss curvature flow} is well-defined and is equivalent to an anisotropic Gauss curvature flow 
\begin{eqnarray}\label{cap-GCF-2}
  \left( \p_t X(x,t)\right)^\perp=-\ell K (x,t)\nu(x,t), 
\end{eqnarray}
where 
\begin{eqnarray}\label{ell}
\ell \coloneqq 1+\cos\theta \< e, \nu\>.\end{eqnarray}
We will prove in the paper that starting from a strictly convex capillary hypersurface, flow \eqref{c-Gauss curvature flow} shrinks to a point in finite time and, after normalization, converges to a soliton. Namely,  a solution to
\begin{eqnarray}\label{soliton0}
 \left\{
\begin{array}{llll}
	K   &=& \ell^{-1} {\<X,\nu\> }  & \text{ in } M ,\\ 
	\<\nu,e\>&=& -\cos\theta & \text{ on } \p M ,\end{array} \right.
\end{eqnarray}
when $\theta \in (0,\pi/2)$.
That is, we establish the capillary counterpart of the results by Tso \cite{Tso85} and Guan–Ni \cite{GN}.

 \begin{thm}\label{main-thm} 
 Let $\S_0$ be a smooth strictly convex capillary hypersurface in $\ol{\R^{n+1}_+}$ and $\theta\in(0,  \pi /2)$, then  flow \eqref{c-Gauss curvature flow} has a smooth solution $\S_t\coloneqq X(\cdot,t)$, which is strictly convex capillary hypersurface for all $t\in[0,T^*)$, where $T^*\coloneqq \frac{ {\rm{Vol}}(\widehat{\S}_0)}{ (n+1){\rm{Vol}}(\widehat{\C_{\theta}})}  .$  As $t\rightarrow T^{\ast}$, then $\S_t$ shrink to a point $p\in \p{\RR^{n+1}_{+}}$.  Moreover, the rescaled capillary hypersurfaces 
 \begin{eqnarray}\label{rescaled-hypersurface}
  \widetilde{\S}_{t}\coloneqq \left(\frac{ {\rm{Vol}}(\widehat{\C_{\theta}})}{ {\rm{Vol}}(\wh\S_t )}\right)^{\frac{1}{n+1}}(\S_{t}-p),
 \end{eqnarray}
 converge to a smooth, strictly convex capillary hypersurface $\S_\infty\subset \ol{\RR^{n+1}_+}$ as  $t\to +\infty$, which is a solution to \eqref{soliton0}. 
\end{thm}

It is easy to check that the spherical cap $\C_{\theta}$ 
is a solution of Eq. \eqref{soliton0}. It is very interesting to classify all soliton solutions to \eqref{soliton0}. We propose the following conjecture,  serving as the capillary counterpart of the results in \cite{And99} and \cite{CD}.

\
\begin{conjecture}
  {\it   For $\theta\in(0,\pi)$, the spherical cap $\C_{\theta}$ is the unique strictly convex solution to Eq. \eqref{soliton0}.}
\end{conjecture}

\

Remark that if $n=1$ and if we further assume that $\S$ is a symmetric (along $\{(x_1,x_2)\in \RR^2_+ \mid x_1=0\}$) curve in $\overline{\RR^2_+}$, then a result 
in \cite[Theorem~5.2]{BLYZ2012} implies that the conjecture is true.

 It is clear that if this conjecture is true, our result  Theorem \ref{main-thm} implies  that the rescaled capillary hypersurfaces as in \eqref{rescaled-hypersurface} 
converge to a uniquely determined spherical cap, at least for $\theta <\pi/2$.

 Our flow \eqref{c-Gauss curvature flow} or \eqref{cap-GCF-2} is closely related to the anisotropic Gauss curvature flow.
 Characterizing the asymptotic behavior of the evolution starting from a  closed convex hypersurface under the anisotropic Gauss curvature flow
\begin{eqnarray}\label{aniso-G-flow} 
\p_t X = -f(\nu)K \nu, \end{eqnarray}
is always a challenging problem. Here $f(\nu)$ is the anisotropic effect function being evaluated at the unit normal vector $\nu$. Andrews \cite{And00} proved that such a flow will contract any smooth, convex, closed hypersurface to a point in finite time. Very recently, B\"{o}r\"{o}czky-Guan \cite{BG}  introduced an entropy functional for flow \eqref{aniso-G-flow}, and established the weak convergence of flow \eqref{aniso-G-flow} and applied it to study the weak solution of the logarithmic Minkowski problem. In the smooth category,  the logarithmic Minkowski problem reduces to
\begin{eqnarray}\label{log-mink}
    h\det(\n^2 h+h\s)=f, ~~~{\rm{in}}~\SS^n, 
\end{eqnarray}  
which is the soliton equation associated with flow \eqref{aniso-G-flow}. It is a well-known conjecture (log-Brunn–Minkowski conjecture) to prove the uniqueness of the solution to \eqref{log-mink} when $f$ is an even, positive smooth function. When $n=1$, this conjecture was confirmed by Gage \cite{Gag93} for symmetric convex bodies in $\RR^2$.
For a general $f$, Yagisita \cite{Yag} constructed two different positive solutions to \eqref{log-mink} for a non-even function $f$ in $\SS^1$. Very recently, Chen–Feng–Liu \cite{CFL} established the uniqueness of solutions to Eq.~\eqref{log-mink} for $n=2$ under the assumption that $f$ is sufficiently close to $1$ in the $C^\alpha$ norm, without requiring evenness. This result was extended to higher dimensions by B\"ar\"oczky–Saroglou \cite{BS24}. By viewing $f=\ell$ in \eqref{log-mink}, we observe that Conjecture 1.2 aligns with the log-Brunn-Minkowski conjecture, except for the presence of an additional boundary condition. For further developments related to \eqref{log-mink}, we refer to a recent survey by B\"or\"oczky \cite{Boroczky2023survey} and the references therein, including \cite{BLYZ2012, BLYZ2013, BS24, CFL, CHLL, KM2022, Sar2015}, among others.

In our setting, it is easy to observe that flow \eqref{cap-GCF-2} is equivalent to \eqref{aniso-G-flow} by viewing $$f(\nu)\coloneqq 1+\cos\theta \< \nu,e\>,$$ except for the presence of an additional boundary condition. From this perspective, the capillary Gauss curvature flow in a Euclidean half-space is closely related to the study of the corresponding anisotropic Gauss curvature flow.

\

Now we outline the main ideas of the proof.
First of all, we establish a Blaschke-Santaló theory for capillary hypersurfaces by suitably generalizing concepts from convex geometry, which is of independent interest and plays a crucial role in the introduction of a capillary entropy later. Then we split the statement of Theorem \ref{main-thm} into two parts,  Theorem~\ref{main-existence} and Theorem \ref{GCF-normal to soliton}. In the first part,  by using the strategy presented in Tso \cite{Tso85} (see also \cite[Section~15.3]{ACGL}), we prove that flow \eqref{c-Gauss curvature flow} shrinks to a point in finite time, as stated in Theorem \ref{main-existence}.  The comparison principle implies that the finite time existence of flow \eqref{c-Gauss curvature flow}, and then assuming that the inner capillary radius of evolving capillary bodies has a uniformly positive lower bound, we can establish a priori estimates for the solution of flow \eqref{c-Gauss curvature flow} in $[0, T^{\ast})$. In our case, the difficulties arising from boundaries need to be addressed. Therefore, we introduce some new auxiliary functions, which are well adapted to capillary problems. 
To derive the principal curvatures bound along flow \eqref{c-Gauss curvature flow}, we introduce a new test function
\begin{eqnarray*}
    \varphi \coloneqq \frac{K}{u-c_0}, 
\end{eqnarray*}where $u\coloneqq \ell^{-1}{\<X,\nu\>}$ and $c_0>0$ is a constant. Here, $u$ is the capillary support function. A key observation is that it satisfies a homogeneous Neumann boundary condition along $\partial \C_\theta$, see Eq. \eqref{neumann-tso-fun}, a similar idea is also used in other problems related to capillary hypersurfaces, see \cite{MWW2024, 
SW, WWX23}. It turns out that the two-sided estimates for the Gauss curvature hold for all $\theta \in (0, \pi)$. The range $\theta < \pi/2$ will be used to derive the two-sided estimates of the principal curvatures, in particular to ensure that Eq.~\eqref{boundary of Phi} holds along $\p\C_\theta$. Next, we transform flow \eqref{c-Gauss curvature flow} into a parabolic Monge-Amp\`ere equation with Robin boundary condition, and then combining with the a priori estimates results, we derive a contradiction that $T^{\ast}$ is not the maximal existence time for flow \eqref{c-Gauss curvature flow}. This implies that $\text{Vol}(\wh\Sigma_t)$ converges to zero as $t\rightarrow T^{\ast}$. Using the fact that the principal curvatures of $\S_{t}$ have a uniformly positive lower bound, we conclude that flow \eqref{c-Gauss curvature flow} shrinks to a point. 

In the second part, we investigate the asymptotic behavior of flow \eqref{c-Gauss curvature flow} and show Theorem \ref{GCF-normal to soliton}. To begin with, we introduce a capillary entropy $\E_\theta(\wh\S)$ for the capillary convex body $\wh\S\subset \ol{\RR^{n+1}_+}$, which is defined as  
\begin{eqnarray*}
    \mathcal{E}_{\theta}(\wh\S)\coloneqq \sup_{z_0\in {\rm{int}}(\widehat{\partial\S})}    \E_{\theta}(\wh\S,z_0),
\end{eqnarray*} where
\begin{eqnarray*}
    \E_{\theta}(\wh\S,z_0)\coloneqq 
        \frac 1 {\omega_\theta}\int_{\C_\theta} \left(\log u_{z_0}(\xi)\right) \ell(\xi)  d\s,
\end{eqnarray*}
where $\ell(\xi)\coloneqq \sin^2\theta+\cos\theta\<\xi, e\>$, the restriction of $\ell$ defined by \eqref{ell} on $
\C_\theta$, $u_{z_0}(\xi)$ is the capillary support function w.r.t. $z_0$ and $\omega_\theta\coloneqq |\C_\theta|-\cos\theta |\wh{\p\C_\theta}|$ is the capillary area of $\C_\theta$, see Eq. \eqref{capillary support} and Definition \ref{def-cap-entropy} for more details. Then we prove that the capillary entropy is non-increasing along the normalized flow \eqref{GCF-capillary-normalized} in Proposition \ref{prop-mono}, with strict monotonicity unless it is a soliton. This is a capillary analogous result to the closed setting by Guan-Ni \cite[Theorem 3.4]{GN}. Adopting the strategy as in \cite{GN}, we further obtain the estimates
 \begin{eqnarray*} 
         \rho_{+}(\widehat{\S}, \theta)\leq C e^{\mathcal{E}_{\theta}(\widehat\S)},
     \end{eqnarray*}
     and 
     \begin{eqnarray*} 
         \rho_{-}(\widehat{\S}, \theta)\geq C^{-1}{\rm{Vol}}(\widehat{\S}) e^{-n \mathcal{E}_{\theta}(\widehat \S)},
     \end{eqnarray*} where $\rho_{-}(\widehat{\S}, \theta)$ and $\rho_{+}(\widehat{\S}, \theta)$ are the capillary inner and outer radius of a capillary convex body $\widehat\S$ respectively (cf. Sinestrari-Weng \cite[Section 2.2]{SW} or Section \ref{sec2.1} for the precise definition).
  Adapting these estimates to the flow \eqref{GCF-capillary-normalized}, we can prove that the capillary support function $u(x, t)$ admits a positive lower bound, i.e., Theorem \ref{thm-C0}, which is one of the crucial points in the paper and does not hold for a general anisotropic Gauss curvature flow \eqref{aniso-G-flow}, cf. \cite[Section 1]{BG}. Then we prove the uniform $C^{2}$ estimates for the solution of normalized capillary Gauss curvature flow:
  \begin{eqnarray}\label{normalized-cap-GCF}
      \p_t X(x,t)= X(x,t)-K\tilde \nu(x,t),~~~(x,t)\in M\times[0,+\infty),
  \end{eqnarray} 
  or Eq. \eqref{GCF-capillary-normalized}. In deriving a uniform lower bound for the Gauss curvature of the solution to Eq. \eqref{normalized-cap-GCF}, we introduce a new test function
\begin{eqnarray}
P \coloneqq \log(K u^\gamma),
\end{eqnarray}
for some $\gamma > 0$ sufficiently large. The advantage of $P$ is that it satisfies a homogeneous Neumann boundary condition along $\partial \C_\theta$, see Eq.~\eqref{P-neumann}. A similar idea was also used in \cite[Theorem~5.2]{AGN} for the closed case.  
  Along  flow \eqref{normalized-cap-GCF}, the enclosed volume $\wh\S_t$ of $X(M,t)$ with $\p\RR^{n+1}_+$ is preserved. Finally, by combining the monotone property of capillary entropy and the uniform a priori estimates of the solution, we show that the solution of normalized flow \eqref{normalized-cap-GCF} converges to a self-similar solution of flow \eqref{c-Gauss curvature flow}, namely Eq. \eqref{soliton0}.

  \

\noindent{\it Organization of the rest of the paper.} 
In Section \ref{sec2}, we first collect some basic properties of the capillary convex body. Subsequently,  we introduce the concept of a capillary polar convex body and establish a capillary Blaschke-Santal\'o inequality. In Section \ref{sec3}, we prove that the flow \eqref{c-Gauss curvature flow} will shrink to a point. Section \ref{sec4} is devoted to investigating the asymptotic behavior of flow \eqref{c-Gauss curvature flow} and proving that the normalized flow will converge to the self-similar solution of flow \eqref{c-Gauss curvature flow}.  Section \ref{sec5} provides a discussion on certain generalizations to the capillary $\a$-power Gauss curvature flow. We introduce the concept of $\theta$-capillary convex bodies by removing the smoothness assumption and define its corresponding capillary entropy functional as an appendix in the last Section \ref{sec-6}.
 
\section{Preliminaries on  convex capillary  hypersurfaces}\label{sec2}
First, we collect some basic facts about capillary convex bodies. Subsequently, we introduce the concept of duality for capillary convex bodies and establish a Blaschke–Santal\'o type inequality, which serves as the capillary analogue of the classical Blaschke–Santal\'o inequality.

\subsection{Capillary hypersurface}\label{sec2.1}
Assume $\S\subset\ol{\RR^{n+1}_+}$ is a capillary hypersurface in the sense of the definition in Section \ref{sec1}. Let $\mu$ be the unit outward co-normal of $\p\S$ in $\S$ and  $\overline{\nu}$ the unit outward normal to $\partial\Sigma$ in $\partial\ov{\mathbb{R}^{n+1}_+}$ such that $\{\nu,\mu\}$ and $\{\overline{\nu},e\}$ have the same orientation in the  normal bundle of $\partial\Sigma\subset\ov{{\R}^{n+1}_+}$. Namely, along $\partial\S$, there holds\begin{eqnarray}\label{co-normal bundle}
		\begin{array} {rcl}
			e &=&\sin\theta \mu-\cos\theta \nu,
			\\
			\overline{\nu} &=&\cos\theta \mu+\sin\theta \nu.
		\end{array}
\end{eqnarray}
 Let $\widehat{\S}$ be the bounded closed region in $\overline{\mathbb{R}^{n+1}_+}$ enclosed by $\S$ and $\partial \mathbb{R}^{n+1}_+$. The boundary of $\widehat{\S}$ consists of two components: one is $\Sigma$, and the other, denoted by $\widehat{\partial \Sigma}$, lies on $\partial \mathbb{R}^{n+1}_+$. We call such $\wh\S$ the \textit{capillary convex body} and denote $\mathcal{K}_\theta$ as the set of all capillary convex bodies in $\ol{\mathbb{R}^{n+1}_{+}}$, and $\mathcal{K}^{\circ}_{\theta}$ as the family of capillary convex bodies for which the origin is an interior point of their flat boundary.

We collect the notion of support function and capillary support function for the capillary convex bodies in $\K_\theta$.
\begin{defn}[Capillary support function]\label{defn-capillary-support}
Let $\wh\Sigma\in \K_\theta$, the support function $h_{z}$ of ${\S}$ with respect to $z\in\widehat{\S}$ is defined by
 \begin{eqnarray*}
      h_{z} (\xi)\coloneqq   \max _{X\in \Sigma}\< \xi-\cos\theta e, X-z\>, \quad \xi\in \C_{\theta}.
	\end{eqnarray*}
The capillary support function of $\S$ with respect to $z$ is defined by
 \begin{eqnarray}\label{capillary support}
     u_{z}(\xi)\coloneqq \frac{h_{z}(\xi)}{\ell (\xi)}, ~~~~\xi\in \C_\theta.
 \end{eqnarray}
\end{defn}
Without confusion, we will abbreviate $h$ and $u$ as the support function and the capillary support function with respect to the origin, respectively. Otherwise, we will write them explicitly with respect to an associated point.
One may refer to \cite[Remark~2.3]{MWW} for the geometric meaning of the capillary support function $u_{z}$. 

Given any strictly convex capillary hypersurface $\S$ with $\wh\S\in \K_\theta^\circ$, we can parametrize it by the inverse capillary Gauss map \eqref{cap Gauss map}, namely
\begin{eqnarray}
    X(\xi)\coloneqq \tilde \nu^{-1}(\xi)\colon  &\C_{\theta} &\rightarrow \S,\notag \\
& \xi &\mapsto \n h(\xi)+h(\xi)(\xi-\cos\theta e) ,\label{capillary-Gauss-map-para}
\end{eqnarray} (cf. \cite[Lemma~2.4]{MWWX}).
The principal radii of $\S$ at $X(\xi)$
 are given by the eigenvalues of $(\n^{2}h+h\sigma)$
 with respect to the round metric $\sigma$ of $\C_\theta$.  On the other hand, for any function $h\in C^{2}(\C_{\theta})$  that satisfies 
\begin{eqnarray*}
    \n^{2}h(\xi)+h(\xi) \sigma >0, ~~~ \forall ~\xi\in\C_\theta,
\end{eqnarray*}
and
\begin{eqnarray*} 
    \n_{\mu}h=\cot\theta  h, \quad \text{on  }~\partial\C_{\theta},
\end{eqnarray*}
then $h$ is the support function of a capillary convex body $\wh\S\in \K_\theta$, cf. \cite[Proposition 2.6]{MWWX}.

To derive a priori estimates for flow \eqref{c-Gauss curvature flow} in the next section, it is convenient to use the notion of the capillary inner/outer radius of $\widehat\Sigma$ (see, e.g.,   
\cite[Section 2.2]{SW}) and is a natural adaptation of the classical notion for a convex closed hypersurface. For a convex capillary hypersurface $\S\subset \ol{\RR^{n+1}_+}$ and  all fixed $\theta\in(0,\pi)$, the capillary inner radius of $\wh\S$ is defined as

\begin{eqnarray*}  
\rho_{-}(\widehat\Sigma, \theta)\coloneqq \sup \left\{r>0 ~\mid ~ \widehat{\C_{r,\theta}(x_0)}\subset \widehat\S\text{ for some } x_0\in \p  \ov{\RR^{n+1}_+}\right\},\end{eqnarray*}and 
 the capillary outer radius of $\wh\Sigma$ is defined   as
\begin{eqnarray}\label{defn-capillary-outer-radius}  
\rho_+(\widehat\Sigma, \theta)\coloneqq \inf \left\{r>0 ~\mid ~\widehat\S\subset \widehat{\C_{r,\theta}(x_0)} \text{ for some } x_0\in  \p \ov{\RR^{n+1}_+} \right\},
\end{eqnarray}where 
\begin{eqnarray*}
    \C_{r,\theta}(x_0)\coloneqq \{x\in \ov{\RR^{n+1}_+} ~\mid~ |x-(x_0+r\cos\theta  e)|=r\}
\end{eqnarray*} 
is the spherical cap centered at the point $x_0+r\cos\theta e$ with the radius $r>0$.  

\begin{rem}\label{rem-cap-inner-radius}
From the definition of $\rho_-(\wh\S,\theta)$, it is easy to see that if the capillary inner radius of $\wh\S$ has a positive lower bound $\varepsilon_0$  for some $\varepsilon_0>0$, then its associated capillary support function \eqref{capillary support} satisfies $$u_{z_0}(\xi)\geq \varepsilon_0, ~~\text{ for all } \xi\in \C_\theta,$$ for some $z_0\in \text{int}(\wh{\p\S})$ (see also \cite[Remark 2.3]{MWW}). 

\end{rem}

Given a convex  body $\wh\S\subset \RR^{n+1}$, recall that we have the classical notion of the inner radius of $\widehat{\S}$, which is defined as   
\begin{eqnarray*}    \rho_{-}(\widehat{\S})\coloneqq \sup \left\{\rho>0 \mid B_{\rho}^{+}(x_{0})\subset \widehat{\S}~\text{for~some~}x_{0}\in \partial \ov{\RR^{n+1}_{+}} \right\},\end{eqnarray*} and the outer radius of $\S$ is defined as
\begin{eqnarray*}    \rho_{+}(\widehat{\S})\coloneqq \inf \left\{\rho>0 \mid \widehat{\S}\subset B_{\rho}^{+}(x_{0})~\text{for~some~}x_{0}\in \partial \ov{\RR^{n+1}_{+}} \right\}\end{eqnarray*} where $B^{+}_{\rho}(x_{0})\coloneqq B_{\rho}(x_{0})\cap \ol{\RR^{n+1}_+}$ and $B_{\rho}(x_0)$ is an open ball of radius $\rho$ centered at $x_0$ in $\RR^{n+1}$.

We next show that for capillary convex bodies, the inner (outer) radius and the capillary inner (outer) radius can be mutually controlled if $\theta\in(0,\pi/2)$.

\begin{prop}
    Let $\widehat{\S}\in \K_{\theta}$ and $\theta \in (0, \pi/2)$. Then there hold
    \begin{eqnarray}\label{inner radius}
         (1-\cos\theta )\rho_{-}(\widehat{\S}, \theta)\leq \rho_{-}(\widehat{\S})\leq \sin\theta \rho_{-}(\widehat{\S}, \theta),
    \end{eqnarray}
    and 
    \begin{eqnarray}\label{outer radius}
          (1-\cos\theta )\rho_{+}(\widehat{\S}, \theta)\leq \rho_{+}(\widehat{\S})\leq \sin\theta \rho_{+}(\widehat{\S}, \theta),
    \end{eqnarray}
\end{prop}

\begin{proof}
  We begin by proving that \eqref{inner radius} holds.
 From the definition of capillary inner radius, there exists a spherical cap $\C_{r_{0}, \theta}(x_{0})$ with radius $r_{0}=\rho_{-}(\widehat{\S}, \theta)$, such that $\widehat{\C_{r_{0}, \theta}(x_{0})}$ is contained in $\widehat{\S}$. It is obvious that $$B_{(1-\cos\theta)r_{0}}^{+}(x_{0})\subseteq \widehat{\C_{r_{0}, \theta}(x_{0})}\subseteq \widehat{\S},$$
 which implies 
    \begin{eqnarray*}
        (1-\cos\theta )\rho_{-}(\widehat{\S}, \theta)\leq \rho_{-}(\widehat{\S}).
    \end{eqnarray*}
On the other hand, from \cite[Proposition 2.4, Eq. (2.26)]{SW}, we get
\begin{eqnarray*}
    \rho_{-}(\widehat{\S}, \theta)\geq \frac{\rho_{-}(\widehat{\S})}{\sin\theta}.
\end{eqnarray*}
Hence, this completes the proof of \eqref{inner radius}. Similarly, by applying \cite[Proposition 2.4, Eq.~(2.27)]{SW}, the proof of \eqref{outer radius} follows in an analogous manner.
\end{proof}

\subsection{Capillary Blaschke-Santal\'o inequality}
In this subsection, we define the polar body $\wh\S^*_{z}$ of the capillary convex body $\widehat{\S}$ with respect to the point $z\in \text{int}(\widehat{\partial \S})$. Subsequently, we prove a non-sharp version of capillary Blaschke-Santal\'o inequality, which concerns the upper bound estimate of the Mahler type volume $$\inf\limits_{z\in {\rm{int}}(\widehat{\partial\S})}{\rm{Vol}}(\widehat{\S}) {\rm{Vol}}(\widehat{\S}^{*}_{z}).$$ Such a kind of estimate inequality will be used to show that a capillary type entropy functional, as in Firey \cite{Firey} (see also Guan-Ni \cite{GN}) is well-defined, we postpone this to Section \ref{sec4}.  

First of all, we recall the capillary gauge function $F_\theta\colon \RR^{n+1}\to [0,+\infty)$ defined by
\begin{eqnarray*}
    F_\theta(\xi) \coloneqq |\xi|+\cos\theta \<e,\xi\>,
\end{eqnarray*} see, e.g., \cite[Section 3]{LXZ}. The dual gauge function $F_\theta^\circ\colon \RR^{n+1}\to \RR$ of $F_\theta$ is given by 
\begin{eqnarray*}
    F^\circ_\theta (x)=\frac 1{\sin \theta} \left(\sqrt {|x|^2 +\cot^2 \theta \<e,x\>^2} -\cot \theta \<e,x\>\right).
\end{eqnarray*}
It is easy to see that $\left\{x\in \ol{\RR^{n+1}_+}\mid  F_\theta^\circ(x)=r\right\}$ is exactly the capillary spherical cap with radius $r$ and centered at $\cos\theta re$, namely $\C_{r,\theta}(o)$.
Denote $$\mathcal{D}_{\theta}(r)\coloneqq \left\{\eta=(\eta', \eta_{n+1})\in {\RR^{n}\times \RR} ~\mid  F_{\theta}(\eta)=r \text{ and } \eta_{n+1}\ge r\frac{ \cos \theta}{\sin^2\theta} \right\},$$ which turns out to be a portion of the ellipsoid given by  \begin{eqnarray*}   \frac{|\eta'|^2}{r^2/\sin^2\theta}+\frac{\left(\eta_{n+1}-\frac{r\cos\theta}{\sin^2\theta}\right)^2}{r^2/\sin^4\theta}=1.\end{eqnarray*} For simplicity, we denote $\DD_{\theta}=\DD_{\theta}(1)$.

It is easy to check
for any $\xi\in \C_{\theta}$, we have
\begin{eqnarray}\label{eq:gradF} 
\n F^{\circ}_{\theta}(\xi) =\frac{1}{\sin \theta} \frac{\xi-\cos\theta F^{\circ}_{\theta}(\xi) e}{\sqrt{|\xi|^2 +\cot^2\theta \<e,\xi\>^2}}.
\end{eqnarray}
It is natural to introduce a capillary adaptation of the Cahn-Hoffman map $\Psi\colon \C_\theta\to \RR^{n+1}$ as 
\begin{eqnarray}\label{Cahn-Hoffman}
    \Psi(\xi)\coloneqq\n F_\theta ^\circ(\xi)=\frac{\xi-\cos\theta e}{F_\theta(\xi-\cos\theta e)}=\frac{\xi-\cos\theta e}{\ell},
\end{eqnarray} 
which is the restriction of $\n F_\theta$ on $\C_\theta$.
The map $\Psi$ satisfies the following property.
\begin{prop}
$\Psi$ is a diffeomorphism  from $\C_{\theta}$ to $\mathcal{D}_{\theta}$.
\end{prop}

\begin{proof}
From \eqref{Cahn-Hoffman}, we have 
\begin{eqnarray*}
    F_{\theta}(\Psi(\xi))=F_{\theta} \left(\frac{\xi-\cos\theta e}{F_{\theta}(\xi-\cos\theta e)} \right)=1,
\end{eqnarray*}
and for $\xi=(\xi',0)\in \partial {\mathcal C}_\theta$, there holds
$$
\Psi (\xi) =\ell^{-1}(\xi) (\xi -\cos \theta e)=\sin^{-2}\theta(\xi',\cos \theta),
$$
which implies $\Psi$ is a map from $\C_{\theta}$ to $\mathcal{D}_{\theta}$.

 For all $\xi\in \C_{\theta}$,  direct calculations yield
   $$
D\Psi (\xi)=\left (\begin{array}{cccc}
\frac{1}{\sin^{2}\theta+\cos\theta \<\xi, e\>} &\dots&  0 & \frac{\cos\theta\xi_{1}}{(\sin^{2}\theta+\cos\theta\<\xi,e\>)^{2}} \\
 \vdots & \vdots& \ddots& \vdots\\
0&  \cdots&\frac{1}{\sin^{2}\theta+\cos\theta\<\xi,e\>} & \frac{\cos\theta \xi_{n}}{(\sin^{2}\theta +\cos\theta\<\xi,e\>)^{2}} \\
0 & \cdots &0 & \frac{1}{(\sin^{2}\theta+\cos\theta\<\xi,e\>)^{2}} \\
\end{array}\right),
$$
then its Jacobi   $\det(D \Psi(\xi))=\frac{1}{\ell^{n+2}(\xi)}>0, ~\forall \xi\in \C_{\theta},$ therefore $\Psi$ is an orientation-preserving diffeomorphism map from $\C_{\theta}$ to $\mathcal{D}_{\theta}$.
\end{proof}

 The following concept is the capillary adaptation of the radial function. 
\begin{defn}[Capillary radial function] \label{def-pola}
Given $\widehat\Sigma\in\K_\theta^{\circ}$, its capillary radial function $\rho_{\wh\S}\colon  \C_\theta\to [0,+\infty)$ is defined by
	\begin{eqnarray*}
	\rho_{\widehat{\S}}(\xi)\coloneqq  \sup \left\{ \l \geq 0\,|\, \l \xi \in \widehat\S  \right\} ,  \quad  \xi \in \C_{\theta }.
\end{eqnarray*}
\end{defn}
Next, we introduce a notion of the capillary polar body for a capillary convex body. We note that this concept differs from the classical one presented in \cite[Section 1.6]{Sch}, where the usual support function $h$ is used. In our setting, we employ the capillary support function \eqref{capillary support}, which is a necessary modification to accommodate the capillary boundary condition.
\begin{defn}[Capillary polar body]
	 Given $\wh\Sigma\in \K_\theta$ and $z\in {\rm{int}}(\widehat{\partial\S})$, we define its capillary polar  body with respect to $z$ by	
     
     $$ \widehat{\S}^{*}_{z}\coloneqq  \left \{X+z\in  \ol{\R^{n+1}_+}  \mid ru_{z}(\xi)\le 1 , \text{ where } X=r\xi \text{ with } \xi\in \C_\theta\right\}, $$ where $u_z$ is the capillary support function of $\S$ with respect to $z$.
  \end{defn} 
\begin{rem}
It can be readily observed that 
$ru_{z}(\xi)\le 1$ is equivalent to   
\begin{eqnarray*}
r{\<\xi-\cos\theta e,y-z\>} \le \ell(\xi) \quad \forall~ y\in \S.
\end{eqnarray*} 
In other words, \begin{eqnarray*}
    r\<\Psi(\xi),y-z\>\leq 1~~~\forall ~ y\in \S,
\end{eqnarray*} where $\Psi$ is given by \eqref{Cahn-Hoffman}. In particular, we observe that the capillary polar body of $\wh\C_\theta$ (w.r.t the origin $o$)  satisfies $(\widehat{\C_{\theta}})^{\ast}_{o}=\widehat{\C_{\theta}}$.

\end{rem}
 For notation consistency, we denote $\Sigma^*_{z}\coloneqq \p ( \widehat{\S}^{*}_{z})\cap  {\RR^{n+1}_+}$. We note that, in general, the capillary polar body of a capillary convex body is not itself a capillary convex body, except in the planar case. Namely, $\S^*_z$ may not be a capillary hypersurface for $n\geq 2$. Nevertheless, we have the following relationship between the capillary support function of a capillary convex body and the radial function of its capillary polar body. This property is of independent interest and may be useful for related problems of capillary hypersurfaces.

\begin{prop}
Given $\wh\S\in \K_\theta^\circ$ and $z\in {\rm{int}}(\widehat{\partial\S})$. Denote $\rho_{z} \coloneqq \rho_{\widehat{\S}_{z}^{\ast}}$ be the capillary radial function of $\widehat{\S}_{z}^{\ast}$. Then there holds 
    \begin{eqnarray*}
    \rho_{z}(\xi)=\frac{1}{u_{z}(\xi)}, ~~~~\forall \xi\in \C_\theta.
    \end{eqnarray*}
\end{prop}
\begin{proof}
  For all $\xi\in\C_\theta$, we see that $\rho_{z}(\xi)\xi\in \S^{*}_{z}-z$.  By the Definition \ref{def-pola} of $\wh\S^*_{z}$, we derive that $\rho_{z}(\xi) u_{z}(\xi)=1$, then the conclusion follows.
\end{proof}

Based on the preceding preparations, we can express the volume of the capillary polar body $\widehat{\S}^{\ast}_{z}$ in terms of the usual support function $h_z$ or the capillary support function $u_{z}$ of $\S$, defined on $\C_\theta$.
\begin{prop}\label{prop-volume-polar-body}
The volume of the capillary polar body $\widehat{\S}^{*}_{z}$ satisfies
 \begin{eqnarray}
{\rm{Vol}}(\widehat{\S}^*_{z}) =
 \frac{1}{n+1}\int _{\C_{\theta }} \frac {\ell(\xi)} {u^{n+1}_{z}(\xi)} d\sigma=\frac{1}{n+1}\int_{\C_{\theta}}\frac{\ell^{n+2}(\xi)}{h^{n+1}_{z}(\xi)}d\sigma, \label{vol}
\end{eqnarray}
where $d\s$ is the standard area element on $\C_\theta$.
\end{prop}

\begin{proof}
From the co-area formula and the  0-homogeneity property of $\n F^\circ_{\theta}$, we compute
\begin{eqnarray}
{\rm{Vol}}(\widehat{\S}^*_{z}) &=& \int_{-\infty}^\infty \int_{\{F^\circ_{\theta}=r\}}\frac {\chi_{\{r u_{z}(\xi)\le 1\}}} {|\n F^\circ_{\theta}(\xi)|} d\sigma dr
=  \int_{-\infty}^\infty \int_{\{F^\circ_{\theta}=1\}} r^n \frac {\chi_{\{ru_{z}(\xi)\le 1\}}}
 {|\n F^{\circ}_{\theta}(\xi)|} d\sigma dr\notag \\
 &=&\int_{\C_\theta} \frac 1{|\n F^\circ_{\theta}(\xi)|}\int_0^{\frac 1{u_{z}(\xi)}} r^n dr d\sigma\notag \\
 &=& \frac 1{n+1} \int_{\C_\theta}  \frac 1{|\n F^\circ_{\theta}(\xi)|}\frac 1{u^{n+1}_{z}(\xi)} d\sigma .\label{polar volume}
\end{eqnarray}
From \eqref{eq:gradF}, we obtain 
\begin{eqnarray}| \n F^{\circ}_{\theta}(\xi )| = \frac{1}{\sin \theta} \frac{1}{\sqrt{|\xi |^2 +\cot^2\theta \<e,\xi\>^2}}= \frac 1 \ell. 
\label{norm gredient}
\end{eqnarray}
Inserting \eqref{norm gredient} into \eqref{polar volume}, the first equality of \eqref{vol} follows, while the second equality follows directly by combining with \eqref{capillary support}. This completes the proof.
\end{proof}

We now introduce the notion of a capillary Santal\'o point $z_s\in \wh{\p\S}$ for a given capillary convex body $\widehat{\S}$, defined as the point for which the infimum of the volume of its capillary polar body is attained.
\begin{defn}[Capillary Santal\'o point]
Given $\wh\S\in \K_\theta$ with $\theta\in(0, \pi/2)$, if there exists a  point $z_{s}=z_s(\wh\S) \in \widehat{\partial\S}$ such that  
$${\rm{Vol}}(\widehat{\S}^*_{z_s})= \inf_{z\in {\rm{int}}(\widehat{\partial \S})} {\rm{Vol}}(\widehat{\S}_z^*),$$ then we call $z_s$ the capillary Santal\'o point of $\widehat{\S}$. 
\end{defn} 

The following proposition justifies that the capillary Santal\'o point is well-defined and unique. The idea of the proof is inspired by \cite[Lemma 2.2]{GN}. A similar argument will be employed again in Proposition \ref{lem entropy} below.

\begin{prop}\label{pro-santalo}
    The capillary Santal\'o point $z_{s}$ is unique and lies in ${\rm{int}}(\widehat{\partial\S})$. Moreover, there holds
\begin{eqnarray}\label{santalo-point-ortho}
    \int_{\C_{\theta}}\frac { \xi_i}{u_{z_s}^{n+2}(\xi)} d\s=0,\quad~ \forall  ~ i=1,\cdots, n,
    \end{eqnarray}where $\xi_i\coloneqq \<\xi, E_i\>$, with $\{E_i\}_{i=1}^n$ denoting the horizontal coordinate vectors of $\RR^{n+1}_+$.
\end{prop}
\begin{proof}
First, we show that the capillary Santal\'o point $z_{s}$ lies in ${\rm{int}}(\widehat{\partial\S})$. We argue by contradiction and assume that $z_{s} \in \partial \S \cap \partial  {\RR^{n+1}_{+}}$. Without loss of generality,  assume $z_{s}$ is the origin. Since $\widehat{\partial \S}$ is a closed convex body in $\partial {\RR^{n+1}_{+}} = \RR^n$, it follows from \cite[Proof of Lemma~2.2, claim (2.3)]{GN} that there exists a unit vector $\zeta \in \RR^{n+1}$ (specifically, $\zeta$ is the unit outward normal of $\partial \S \subset \wh{\p \S}$ at the origin, and for notational convenience, we may assume $\zeta = E_1$,) such that
\begin{eqnarray}\label{L-contained}
L \coloneqq \left\{ -t E_1 \mid  0 < t < t_{0} \right\}\text{ is contained in }\widehat{\partial \S},~\text{ for some small } t_{0}>0.
\end{eqnarray}
Moreover, due to $\theta \in (0,\pi/2)$, the set $\widehat{\S}$ is contained in the half-space $\{x_{1}\leq 0\}$ and is tangent to the hyperplane $\{x_1=0\}$ at the origin. 

For $\xi\coloneqq (\xi_{1}, \xi_{2},\cdots, \xi_{n+1})\in \C_{\theta}$,  define $N(\xi)\coloneqq(-\xi_{1}, \xi_{2}, \cdots, \xi_{n+1})$.
Since $$u_{z_s}(\xi)=\ell^{-1}(\xi)\sup\limits_{X\in \S}\<X, \xi+\cos\theta E_{n+1}\>,$$ for all $\xi\in \C_\theta$, then
 there exists a point $z(\xi)\in \S$, such that 
 $$u_{z_s}(\xi)=\ell^{-1}(\xi)\<z(\xi), \xi+\cos\theta E_{n+1}\>.$$
For all $\xi\in \C_{\theta}$ with $\xi_{1}\geq 0$, by the Definition  \ref{defn-capillary-support}, we have
\begin{eqnarray}
    u_{z_s}\left(N(\xi)\right)&\geq &\ell^{-1}(N(\xi))\<z(\xi), N(\xi)+\cos\theta E_{n+1}\>\notag 
    \\
    &=& \ell^{-1}(\xi)\<z(\xi), \xi+\cos\theta E_{n+1}\>-\ell^{-1}(\xi)\< z(\xi), 2 E_1 \>\notag
    \\
    &\geq& \ell^{-1}(\xi)\<z(\xi), \xi+\cos\theta E_{n+1}\>=u_{z_{s}}(\xi), \label{big}
\end{eqnarray}
where the last line used the fact that  $\<z(\xi), E_1\> \le 0$. Moreover, the above inequality is strict for a subset of $\xi \in \C_{\theta}$ of positive measure.

Consider the function  $$u_{z_t}(\xi)\coloneqq u_{z_s}(\xi)+t\frac{\xi_1}{\ell (\xi)}, ~~~\xi\in \C_\theta, ~~t\in(0, t_0). $$ 
From Definition \ref{defn-capillary-support}, we know that $u_{z_t}$ is the capillary support function of $\widehat{\S}$ with respect to the point $z_{t}=-tE_1$ for all  $0<t<t_{0}$. From \eqref{L-contained}, we have $u_{z_t}(\xi)>0$ for all $t\in(0, t_0)$.  
Using Proposition \ref{prop-volume-polar-body} and \eqref{big}, we compute
\begin{eqnarray*}
    0&\leq& \frac{d}{dt}\left({\rm{Vol}}(\widehat{\S}^{\ast}_{z_{t}})\right)\Big|_{t=0}=-\int_{\C_{\theta}}\frac{\xi_{1}}{u_{z_{s}}^{n+2}(\xi)}d\s\\
    &=&-\left(\int_{\{\xi_{1}>0\}}\frac{\xi_{1}}{u_{z_{s}}^{n+2}(\xi)}d\sigma+\int_{\{\xi_{1}<0\}}\frac{\xi_{ 1}}{u_{z_{s}}^{n+2}(\xi)}d\sigma\right)\\
    &=&-\left(\int_{\{\xi_{1}>0\}}\left(\frac{ \xi_{1}}{u_{z_s}^{n+2}(\xi)}-\frac{ \xi_{1} }{u_{z_s}^{n+2}(N(\xi))}\right)d\s\right)<0,
\end{eqnarray*}
this reaches a contradiction. Hence we get $z_{s}\in   \rm{int}(\widehat{\partial \S})$. 

Next, we move on to prove \eqref{santalo-point-ortho}. 
Let $z_s^{\varepsilon,i} \coloneqq z_{s}+\varepsilon E_{i}$ with $1\leq i\leq n$. For some sufficiently small $\varepsilon_0 > 0$, we have $$z_s^{\varepsilon,i} \in {\rm{int}}(\widehat{\partial\S}) ~~\text{ for all } |\varepsilon| \leq \varepsilon_0,$$ together with Proposition~\ref{prop-volume-polar-body}, we obtain
\begin{eqnarray*}
   0= \frac{d}{d\varepsilon}\left({\rm{Vol}}(\widehat{\S}^{\ast}_{z_s^{\varepsilon,i}})\right)\Big|_{\varepsilon=0} =\int_{\C_{\theta}}\frac{\xi_{i}}{u_{z_{s}}^{n+2}(\xi)}d\sigma, ~~\text{ for all }~~1\leq i\leq n.
\end{eqnarray*} Hence the assertion \eqref{santalo-point-ortho} follows.  

Finally, we prove the uniqueness of the capillary Santal\'o point.  Without loss of generality,  we assume the capillary Santal\'o point $z_s$ is the origin and it lies in ${\rm{int}}(\widehat{\partial\Sigma})$ by the preceding argument. For all $z\in \widehat{\partial\S}$, we rewrite $z\coloneqq (z_1,\cdots, z_n,0)$, and  
\begin{eqnarray}\label{uz}
        u_{z}(\xi)=u(\xi)-  \sum\limits_{i=1}^{n} \frac{z_{i}\xi_{i}}{\ell(\xi)} , \quad \xi\in \C_{\theta},
\end{eqnarray}where $u$ is the capillary support function of $\S$ with respect to the origin. Inserting \eqref{uz} into \eqref{vol} implies
\begin{eqnarray*}
 {\rm{Vol}}(\widehat{\S}_{z}^{\ast})=\frac{1}{n+1}\int_{\C_{\theta}}\frac{\ell(\xi)}{\left(u(\xi)-\ell^{-1}(\xi)\sum\limits_{i=1}^{n}z_{i}\xi_{i}\right)^{n+1}}d\sigma\eqqcolon V(z).
\end{eqnarray*}
A direct calculation shows that 
\begin{eqnarray*}
    \frac{\partial^{2}V(z)}{\partial z_{i}\partial z_{j}}=(n+2)\int_{\C_{\theta}}\frac{\xi_{i}\xi_{j}}{\ell(\xi)u_{z}^{n+3}(\xi)}d\sigma, ~~~\text{ for all } 1\leq i, j \leq n.
\end{eqnarray*} This implies the strict convexity of the function $V(z)$ with respect to $z\in\wh{\p\S}$, which in turn yields the uniqueness of the capillary Santal\'o point. This completes the proof.

\end{proof}

We are now ready to establish the following capillary version of the Blaschke–Santal\'o inequality. Compared to the classical form, this version is non-sharp. Nevertheless, it suffices for our purposes in analyzing the asymptotic behavior of the normalized capillary Gauss curvature flow in Section \ref{sec4}.

\begin{prop}\label{prop-santalo-ineq}
Given $\widehat{\S}\in \K_{\theta}$ and  $\theta\in(0,\pi/2)$, then there holds
   \begin{eqnarray}\label{BS-ineu}
\inf\limits_{z\in {\rm{int}}(\widehat{\partial\S})}{\rm{Vol}}(\widehat{\S}) {\rm{Vol}}(\widehat{\S}^*_{z}) \le  \frac{1}{2}{\rm{Vol}}(\mathbb{B}^{n+1})^{2}.
\end{eqnarray}
\end{prop}
\begin{proof} 
Let $R\colon \ov{\R^{n+1}_+} \to \R^{n+1}$ be the reflection map,  defined by $$ R(x)\coloneqq (x_{1},\cdots, x_{n}, -x_{n+1}), \quad x\in\ov{\RR^{n+1}_{+}}.$$
 Let $V=\S\cup R(\S)$ and $\Omega$ be the bounded domain enclosed by $V$. Then $\O$ is symmetric with respect to the hyperplane $\p\RR^{n+1}_+$. Denote by $\O^{\ast}_z$ the polar body of the convex body $\Omega$ with respect to the point $z \in {\rm{int}}(\O)$, defined by
\begin{eqnarray*}
\O^{\ast}_{z}\coloneqq \left\{X+z\in \RR^{n+1} \mid r \wh h_{z}(\eta)\leq 1, ~ {\rm{where}} ~ X=r\eta ~ {\rm{with}} ~ \eta\in \SS^{n} \right\},
\end{eqnarray*} where $\wh h_z\colon \SS^n\to \RR$ is the usual support function of $\O$ with respect to $z$. From \cite[Section 10.5 and Eq. (10.28)]{Sch} (see also \cite{San}),  
we know there exists a unique point (Santal\'o point) $z_{0}\in \rm{int}(\Omega)$ such that 
\begin{eqnarray*}
    {\rm{Vol}}\left(\Omega_{z_{0}}^{\ast}\right)=\inf \limits_{z\in \rm{int}(\Omega)}{\rm{Vol}}\left(\Omega^{\ast}_{z}\right),
\end{eqnarray*} 
and the Blaschke-Santal\'o inequality  
\begin{eqnarray}\label{class-BS}
    {\rm{Vol}}(\Omega){\rm{Vol}}
    \left(\Omega_{z_{0}}^{\ast}\right)\leq {\rm{Vol}}(\mathbb{B}^{n+1}) ^{2}.
\end{eqnarray}
 From \cite[Eq. (10.22)]{Sch} (or \cite[Lemma 2]{MP}), there holds 
    \begin{eqnarray}\label{sym}
        0=\int_{\SS^{n}}\frac{\eta_{i}}{\widehat h^{n+1}_{z_0}(\eta)}d\eta,\quad 1\leq i\leq n+1.
    \end{eqnarray}
In particular, the symmetry of $\Omega$ with respect to $\p \RR^{n+1}_+$ implies 
    \begin{eqnarray}\label{zs}
        z_{0}\in {\rm{int}}(\O) \cap \p {\RR^{n+1}_{+}}={\rm{int}}({\widehat{\partial \S}}).
    \end{eqnarray}
Since $\ell<1$, from the definition of $\widehat{\S}^{\ast}_{z_{0}}$ and $\O_{z_{0}}^{\ast}$, we deduce that $\widehat{\S}^{\ast}_{z_{0}}\subseteq \O^{\ast}_{z_{0}}$, and hence
\begin{eqnarray*}
    {\rm{Vol}}(\widehat{\S}_{z_{0}}^{\ast})\leq {\rm{Vol}}(\O_{z_{0}}^{\ast}).
\end{eqnarray*}
Together with \eqref{class-BS} and \eqref{zs}, 
  we derive  
\begin{eqnarray*}
    \inf\limits_{z\in {\rm{int}}(\widehat{\partial\S})}{\rm{Vol}}(\widehat{\S}){\rm{Vol}}(\widehat{\S}^{\ast}_{z})&\leq&  {\rm{Vol}}(\widehat{\S}){\rm{Vol}}(\widehat{\S}^{\ast}_{z_{0}})\\
    &\leq&  \frac{1}{2}{\rm{Vol}}(\Omega){\rm{Vol}}
    \left(\Omega_{z_{0}}^{\ast}\right) \\& \leq & \frac{1}{2} {\rm{Vol}}(\mathbb{B}^{n+1})^{2}.
\end{eqnarray*} 
This completes the proof of \eqref{BS-ineu}.
\end{proof}

We conclude this subsection by proposing a conjecture about the optimal capillary Blaschke-Santal\'o inequality. From \eqref{BS-ineu},  it is clear that
\begin{eqnarray*} 
\inf\limits_{z\in {\rm{int}}(\widehat{\partial\S})}{\rm{Vol}}(\widehat{\S}) {\rm{Vol}}(\widehat{\S}^{*}_{z}) \le C(n, \theta) {\rm{Vol}}(\widehat{\C_{\theta }})^2,
\end{eqnarray*}
where $C(n, \theta)>0$ only
depends on $\theta$ and $n$.
Hence, we propose the following conjecture about the optimal capillary Blaschke-Santal\'o inequality.

\begin{conjecture}\label{conj-sharp-capillary-B-S-ineq}  {\it Given $\widehat{\S}\in \K_{\theta}$ and $\theta\in (0,\pi)$, then there holds
\begin{eqnarray*}
\inf\limits_{z\in {\rm{int}}(\widehat{\partial\S})}{\rm{Vol}}(\widehat{\S}) {\rm{Vol}}(\widehat{\S}^{*}_{z}) \le    {\rm{Vol}}(\widehat{\C_{\theta }})  ^2. 
\end{eqnarray*}
Equality holds if and only if $\S$ is a spherical cap. } 
\end{conjecture}

\section{Capillary Gauss curvature flow}\label{sec3}
In this section, we adopt the strategy presented in \cite{Tso85} (see also \cite[Section~15.3]{ACGL}) and show the first part of  Theorem \ref{main-thm}.
Precisely, we establish the following theorem.
\begin{thm}\label{main-existence} 
 Let $\S_0$ be a strictly convex capillary hypersurface in $\ol{\R^{n+1}_+}$ and $\theta\in(0,\pi/2)$, then flow \eqref{c-Gauss curvature flow} has a smooth solution $\S_t\coloneqq X(\cdot,t)$, which is strictly convex capillary hypersurface for all $t\in[0,T^*)$ where $T^*\coloneqq   \frac{ {\rm{Vol}}(\widehat{\S}_0)}{ (n+1){\rm{Vol}}(\widehat{\C_{\theta}})}  .$ Moreover,  the capillary hypersurfaces $\S_t$ shrink to a point $p\in\p{\RR^{n+1}_{+}}$ as    $t\to T^*$.
\end{thm}

We divide the proof of Theorem~\ref{main-existence} into two steps. First, we show that as $t \rightarrow T^{\ast}$ (with $T^{\ast} < +\infty$), the capillary inner radius of $\widehat{\S}_t$ tends to zero. This, in turn, implies that the volume of $\widehat{\S}_t$ vanishes in finite time.
Second, using a contradiction argument, we conclude that as $t \rightarrow T^{\ast}$, the capillary convex body $\widehat{\S}_t$ does not collapse to a lower-dimensional convex set but instead contracts to a point.

\subsection{The scalar flow and finite time existence} \

Using the capillary Gauss map parametrization in  \eqref{capillary-Gauss-map-para}, we transform flow \eqref{c-Gauss curvature flow} into the problem of solving a parabolic Monge-Ampère equation subject to a Robin boundary condition.
\begin{prop}\label{prop-scalar equ}
	Let $\Sigma_t$ be a strictly convex solution to flow \eqref{c-Gauss curvature flow} for $t\in [0, T^*)$ and $\theta\in(0, \pi)$,   then its support function $h(\cdot,t)\colon \C_\theta\to \RR$  satisfies 
\begin{eqnarray}\label{support fun eq. of capillary alpha-GCF} \left\{
\begin{array}{llll}
  \p_t h ~ [\det(\n^2 h+h\sigma)]&=&-(\sin^2\theta+\cos\theta \<\xi,e\>)   & \text{ in } \C_\theta\times[0,T^*),\\ 
	\n_\mu h&=& \cot\theta h & \text{ on } \p \C_\theta \times [0, T^*), \\
 h(\cdot,0) &=& h_0(\cdot)  & \text{ on } \C_\theta, \end{array}  \right.
\end{eqnarray} where  $\mu$ is the unit outward normal of $\p \C_\theta\subset  \C_\theta$ and $h_0$ is support function of $\S_0$.
\end{prop}
\begin{proof}
	To derive the first equation in \eqref{support fun eq. of capillary alpha-GCF}, we adopt the argument as in \cite[Page 869-870]{Tso85}. In fact, consider a family of strictly convex capillary hypersurfaces $\Sigma_{t}$ that satisfies  \eqref{c-Gauss curvature flow}, its support function $h(\xi,t)$ satisfies
	\begin{eqnarray}\label{h evolution GCF}
		\partial_{t}h=\left\< \frac{\partial X}{\partial t},\nu \right\>=-K \< \tilde{\nu},\nu\>,
	\end{eqnarray}
 together with the simple fact that $\nu(X)= \xi-\cos\theta e$ and $\xi\in\C_\theta$, it follows
 $$\<\tilde{\nu}  ,\nu\>=1+\cos\theta \<\nu,e\>
 =\sin^2\theta+\cos\theta \<\xi,e\>.$$  On the other hand, we know $$ K=\frac{1}{\det(\n^2 h+h\sigma)},$$ by inserting this into \eqref{h evolution GCF}, we derive the first equation \eqref{support fun eq. of capillary alpha-GCF} follows. The second equation in \eqref{support fun eq. of capillary alpha-GCF} can be derived by following the same argument in the proof of \cite[Proposition 2.4]{MWW}, and is therefore omitted here.
\end{proof}
 
Notice that \eqref{support fun eq. of capillary alpha-GCF} is a strictly parabolic equation as long as $\n^2 h+h\sigma >0$. 
Hence, by applying the standard parabolic theory with Neumann boundary condition, one can easily obtain the short-time existence for the flow \eqref{support fun eq. of capillary alpha-GCF}. This, in turn,  implies the short-time existence of flow \eqref{c-Gauss curvature flow}. Next, we show that the finite time existence of flow \eqref{c-Gauss curvature flow} (or equivalently flow \eqref{cap-GCF-2}).

\begin{prop} 
Let $\S_0$ be a strictly convex capillary hypersurface in $\ol{\R^{n+1}_+}$ with $\theta\in(0,\pi/2)$, then capillary hypersurface $\S_t$ of flow  \eqref{c-Gauss curvature flow} exists for a finite time $T^*<\infty$. 
\end{prop}
\begin{proof}
Without loss of generality, we assume that the origin lies in the interior of $\text{int}(\wh{\p{\S_0}})$. Then there exists some $r_0>0$ such that $\S_0\subset \widehat{\C_{r_{0}, \theta}}$. Define $$r(t)\coloneqq \left[ r_0^{ n+1}- ( n+1)t \right] ^{\frac{1}{ n+1}},$$ we find that  $ {\C_{r(t), \theta}}$ is a solution to flow \eqref{c-Gauss curvature flow} starting from $\S_{0}=\C_{r_{0}, \theta}$.
By the comparison principle (see, e.g., \cite[Proposition 4.2]{WW20}), we know that $\S_t\subset \widehat{\C_{r(t), \theta}}$.
It is obvious that $r(t)$ converges to zero at a finite time as $t\nearrow \frac{r_0^{n+1}}{n+1}$. Therefore, $\S_t$ cannot exist beyond that time.    
\end{proof}

\subsection{Evolution equations}\ 

In this subsection, we derive evolution equations for various geometric quantities.
For convenience, we introduce a linearized parabolic operator with respect to flow \eqref{support fun eq. of capillary alpha-GCF} 
\begin{eqnarray*}
	\mathcal{L}\coloneqq \partial_{t}-\ell  K \sum_{i,j=1}^n a^{ij} \nabla^{2}_{ij},
\end{eqnarray*}
where $(a^{ij})$ is the inverse matrix of  the principal radii matrix $A\coloneqq (a_{ij})$ with $a_{ij}\coloneqq h_{ij}+h\s_{ij}.$
It is easy to observe that the mean curvature $$H = \sum\limits_{i,j=1}^{n} a^{ij} \delta_{ij}.$$
In what follows, we adopt the Einstein summation convention.

 First, we have the following evolution equations about the support function $h$ and the capillary support function $u$, respectively.
 
\begin{prop}\label{evo-h}
    Along  flow \eqref{support fun eq. of capillary alpha-GCF}, the functions $h$ and $u$ satisfy  
        \begin{eqnarray}\label{h evolution GCF 1}
      \L h= K \ell (hH-n-1),
    \end{eqnarray}  and \begin{eqnarray}\label{h bar evolution GCF}
\L  u= 2 K a^{ij} \n_{i}u  \n_{j}\ell -(n +1)K+  K H u.
    \end{eqnarray}
Along  $\p \C_\theta$, 
  \begin{eqnarray}\label{robin of h}
        \n_\mu h =\cot\theta h,
    \end{eqnarray} 
    and 
     \begin{eqnarray}\label{robin of bar h}
        \n_\mu u =0. 
    \end{eqnarray} 
\end{prop}
\begin{proof} 
First, we prove \eqref{h evolution GCF 1} and  \eqref{h bar evolution GCF}. Let $\{e_i\}_{i=1}^n$ be an orthonormal frame on $\C_\theta$ such that $e_n=\mu$.  Using  Eq. \eqref{h evolution GCF}, we have 
\begin{eqnarray*}
    \L h&=&\p_t h- K \ell a^{ij} \n_{ij} h=-K \ell - K \ell a^{ij} (a_{ij}-h\delta_{ij}).
\end{eqnarray*} 
Since $a^{ij}\delta_{ij}=H$, we conclude that \eqref{h evolution GCF 1} holds.

A direct computation yields
\begin{eqnarray}\label{h bar evo GCF 2}
    \L u&=& \frac{\L h}{\ell} + K\ell a^{ij} \left( \frac{ 2\n_{i}h\n_{j} \ell}{\ell^2} +\frac{ h\n_{i}\n_{j}\ell}{\ell^2}-\frac{2h \n_{i}\ell \n_{j}\ell}{\ell^3}\right).
\end{eqnarray}
Since
\begin{eqnarray}\label{ell ij}
    \n_{i}\n_{j}\ell+\ell \delta_{ij}=-\cos\theta \<\xi-\cos\theta e,e\>\delta_{ij}+\ell \delta_{ij}= \delta_{ij},
\end{eqnarray} substituting \eqref{h evolution GCF 1} and \eqref{ell ij} into     \eqref{h bar evo GCF 2}, we get \eqref{h bar evolution GCF}.

Next, we derive the boundary conditions \eqref{robin of h} and \eqref{robin of bar h}.
Since $D_{e_i}X$ is tangent to $\Sigma$ at $X(\xi)$, we have
  	\begin{eqnarray*}
			\<\nu (X),D_{e_i}X\>=0,\quad 1\leq i\leq n.
		\end{eqnarray*}
		If $\xi\in \p \C_{\theta}$, then $X(\xi) \in \p \ov{\R^{n+1}_+}$. From \eqref{co-normal bundle}, we obtain
		\begin{eqnarray*}
			\<\mu,X\>&=&\< \sin\theta e+\cos\theta \overline{\nu},X\>=\cos\theta \<\overline{\nu} ,X\>\\&=&\cos\theta \<\cos\theta \mu+\sin \theta \nu(X),X\>=\cos^2\theta \<\mu, X\>+\sin\theta\cos\theta h,
		\end{eqnarray*}
        which implies  that
		\begin{eqnarray*}
			\<\mu,X\>=\cot\theta h.
		\end{eqnarray*} 
        Differentiating the support function $h$ gives
		\begin{eqnarray*}
			\n_{{\mu}}h&=&\< D_\mu T^{-1}(\xi ), X\>+\< T^{-1}(\xi ),D_\mu X\>
			\\&=& \<D_\mu\xi,  X\>= \<\mu,X\>,
		\end{eqnarray*}
		where $T^{-1}(\xi)\coloneqq  \nu (X(\xi))$ is the unit normal at $X(\xi)$. 
	Combining the above results yields
		\begin{eqnarray*} 
			\n_{{\mu}}h=\cot\theta h. 
		\end{eqnarray*}
Using \eqref{co-normal bundle} again, we have
  \begin{eqnarray}\notag 
      \n_\mu (\sin^2\theta+\cos\theta\<\xi,e\>)&=& \cos\theta \<\mu,e\>=\cos\theta \sin\theta \\&= & \cot\theta (\sin^2\theta+\cos\theta\<\xi,e\>).\label{robin of xi,e} 
  \end{eqnarray}
Consequently, \eqref{robin of bar h} can be derived from \eqref{robin of xi,e} and \eqref{robin of h}.
\end{proof}

The evolution equation of the Gauss curvature $K$ will be useful for our subsequent analysis.
\begin{prop}\label{evo of K}
    Along  flow \eqref{support fun eq. of capillary alpha-GCF}, there holds
  \begin{eqnarray}\label{K evol GCF}
      \L K= 2 Ka^{ij}\n_{i}K \n_{j}\ell+K^{2} H.
  \end{eqnarray}
On $\p\C_\theta$,
\begin{eqnarray}\label{robin of K}
        \n_\mu K  =0 .  
    \end{eqnarray}  
\end{prop}
\begin{proof}
Direct calculations yield
\begin{eqnarray*}
   \L K &=&  \p_t K-K \ell a^{ij}\n_{ij} K\\&=&-Ka^{ij} (\n_{ij}(\p_t h)+\p_t h\delta_{ij})- K\ell a^{ij}\n_{ij} K
    \\&=& K a^{ij}(\n_{ij}(K\ell)+K\ell \delta_{ij})- K \ell a^{ij}\n_{ij} K
    \\&=&2 K a^{ij} \n_{i}K\n_{j}\ell  +K^{2} H,
\end{eqnarray*}where \eqref{ell ij} has been used in the last equality. 
To show \eqref{robin of K}, by taking the time derivative on \eqref{robin of h} along $\p\C_\theta$ we have
\begin{eqnarray*}
    \n_\mu (\p_t h)=\cot\theta \p_t h.
\end{eqnarray*}
In view of $$\p_t h=-K \left(\sin^2\theta+\cos\theta \<\xi,e\>\right),$$ together with \eqref{robin of xi,e}, then  we have \eqref{robin of K} on $\p\C_\theta$.
\end{proof}

\subsection{Curvature estimates}\ 
 
In this subsection, we establish curvature estimates under the assumption that the capillary inner radius $\rho_-(\wh\S_t,\theta)$ is bounded below by a positive constant. The approach is inspired by Tso \cite{Tso85} and Guan-Ni \cite{GN}. We remark that the two-sided estimates for the Gauss curvature hold for all contact angles $\theta \in (0, \pi)$.

First of all,  we prove that the Gauss curvature of the solution $\S_t$ to flow \eqref{c-Gauss curvature flow} is uniformly bounded from below. 
\begin{prop}\label{lower bound of K}
    Let $K_0$ be the Gauss curvature of initial strictly convex capillary  hypersurface $\S_0$ with $\theta\in(0,\pi)$, and $\S_t$ solves \eqref{c-Gauss curvature flow} on $[0,T^*)$.  Then there holds
    \begin{eqnarray}\label{lower bdd K}
        K (\xi, t) \geq \min_{\C_\theta} K_0,
    \end{eqnarray} for all $(\xi,t)\in \C_\theta \times [0,T^*)$.
\end{prop}
\begin{proof}
    
From Proposition \ref{evo of K},  we see 
\begin{eqnarray*}
    \L K \geq 0 ~~~\text{ mod } \n K ,~~ \text{ in } \C_\theta,
\end{eqnarray*}
and $\n_\mu K=0$ on $\p \C_\theta$. Then \eqref{lower bdd K} follows directly from the maximum principle.
\end{proof}

Next, if the capillary inner radius of $\wh\S_t $ has some positive lower bound, we show that the Gauss curvature of the solution $\S_{t}$ to flow \eqref{c-Gauss curvature flow} is uniformly bounded from above. In order to prove it, we introduce a new test function
\begin{eqnarray*}
    \varphi\coloneqq \frac{K}{u-c_0},
\end{eqnarray*}where $u$ is the capillary support function and $c_0\coloneqq \frac 1 2 \inf\limits_{\C_\theta\times[0, T^*)} u(\xi,t)$. The function $\varphi$ is a natural adaptation of the one used by Tso \cite{Tso85} in the capillary setting, see also a similar idea in \cite[Section 3.2]{SW} for a mean curvature type flow. A key observation for us to introduce this new test function is that it satisfies a homogeneous Neumann boundary value condition on $\p\C_\theta$, see Eq. \eqref{neumann-tso-fun}.

\begin{prop}\label{upper bound of K}
Let $\S_t$ be the solution of flow \eqref{c-Gauss curvature flow}  on $[0,T^*)$ with $\theta \in (0,\pi)$. If the capillary inner radius $\rho_-(\wh\S_t,\theta)$ of $\S_t$ has a uniform positive lower bound $\varepsilon_0>0$,  then there exists a positive constant $C$, depending only on $\frac{1}{\varepsilon_0}$ and the initial data, such that
    \begin{eqnarray*}
  K(\xi,t) \leq C ,
    \end{eqnarray*}
for all $(\xi,t)\in \C_\theta \times [0,T^*)$.
\end{prop}

\begin{proof} 
First,  if the capillary inner radius of $\wh\S_t $ has a uniformly positive lower bound $\varepsilon_{0}$ for all $t\in[0, T^*)$, then according to  Remark \ref{rem-cap-inner-radius},
there exists a point $z_{0}\in \text{int}( \widehat{\partial\S_{t}})$, such that $u_{z_0}(\xi, t)\geq \varepsilon_0 $ for all $(\xi,t)\in\C_\theta\times [0, T^*)$. Without loss of generality, we assume that $z_{0}$ is the origin  and  consider the function    \begin{eqnarray*}
    \varphi(\xi,t)\coloneqq \frac{K}{u-c_0},
\end{eqnarray*} where $c_0\coloneqq \frac{\varepsilon_0} 2$. Thus, $\varphi$ is well-defined for all $(\xi,t)\in \C_\theta\times[0, T^*)$.

Direct computation yields 
    \begin{eqnarray}\label{var euqation}
        \L \varphi &=&\frac{\L K}{u-c_0} -\frac{K}{(u-c_0)^2} \L u+ \frac{2 K\ell}{u-c_0} a^{ij} \n_{i}\varphi\n_{j} u.
    \end{eqnarray}
    Inserting Eq. \eqref{h bar evolution GCF} and \eqref{K evol GCF} into \eqref{var euqation}, we have
    \begin{eqnarray*}
        \L \varphi &=& \frac{ 2K a^{ij}\n_{i}K\n_{j}\ell+ K^{2} H}{u-c_0} -\frac{K}{(u-c_0)^2} \left(-(n+1)K+K H u\right)  \\&& -\frac{2K}{(u-c_0)^2} K a^{ij} \n_{j}u \n_{i}\ell + \frac{2 K \ell}{u-c_0} a^{ij} \n_{i} \varphi \n_{j} u        \\&=&  2  K a^{ij} \n_{i}\varphi\n_{j}\ell+2 \varphi \ell a^{ij} \n_{i}\varphi\n_{j} u+ \varphi^2 (n +1-c_0  H)\\&\leq & n  c_0^{1+\frac{1}{n  }}\varphi^2\left(-\varphi^{\frac{1}{n  }} +\frac{n +1}{n  c_0^{1+\frac 1 {n}}}\right),\quad \text{ mod } \n \varphi,
    \end{eqnarray*}where in the last inequality we have used the fact that $H\geq n K^{\frac 1 n}$ and $ u\geq 2c_0$. 
   On the other hand, in view of \eqref{robin of bar h} and \eqref{robin of K}, we have 
    \begin{eqnarray}\label{neumann-tso-fun}
        \n_\mu \varphi=  0,\quad \text{ on } \p \C_\theta.
    \end{eqnarray}
    By the maximum principle, we have
    \begin{eqnarray*}
        \varphi \leq \max\left\{\frac{1}{c_0} \left( \frac{n +1}{ nc_0}\right)^{n}, \max_{\C_\theta}\varphi(\cdot, 0) \right\} .
    \end{eqnarray*}
   It follows
\begin{eqnarray}\label{K-upper-bound}
        K\leq \max\left\{\left(\frac{n+1}{n}\right)^{n}\frac{\max\limits_{\C_{\theta}}u(\cdot, t)}{c_{0}^{1+n}}, \frac{ \max\limits_{\C_{\theta}}u(\cdot, t)\cdot K_{0}}{u(\cdot,0)-c_{0}}
        \right\}.
    \end{eqnarray}
Since $\p_t u \leq 0$ and $\n_\mu u =0$ on $\p \C_\theta$,  we have $$u\leq  \max\limits_{\C_\theta} u(\cdot, 0).$$  Therefore, as long as  flow \eqref{c-Gauss curvature flow} exists and $ u\geq 2c_0$, there holds
\begin{eqnarray*}
        K   \leq  \max\limits_{\C_\theta} u(\cdot, 0) \cdot  \max\left \{\frac{1}{c_0} \left( \frac{n+1}{n c_0}\right)^{n}, \frac{\max\limits_{\C_\theta}K_0}{c_0}\right\} .
    \end{eqnarray*}Thus the statement follows.
\end{proof}

Finally, if the capillary inner radius remains positive, we establish the positive lower and upper bounds of the principal radii. This, in turn, implies that flow \eqref{c-Gauss curvature flow} (also equivalently \eqref{support fun eq. of capillary alpha-GCF}) remains uniformly parabolic on any time interval and the strict convexity is preserved along flow \eqref{c-Gauss curvature flow}.

\begin{prop}\label{curvature est}
Let $\S_t$ be the solution of the flow \eqref{c-Gauss curvature flow} on the time interval $[0, T^*)$. If $\theta\in (0,\pi/2)$ and the capillary inner radius of $\S_t$ has a uniform positive lower bound $\varepsilon_0$,  then there exist positive constants $C_0$ and $C_1$ such that 
    \begin{eqnarray}\label{upper-lower-curvature est}
      C_0^{-1}\leq \k_i\leq  C_1 \quad \text{ for all } 1\leq i\leq n,
    \end{eqnarray} 
    where the constant $C_1$ only depends  on $n, \frac{1}{\varepsilon_0}$ and  initial data. 
    In particular,   the constant $C_0$ only depends on the initial data and $n$. 
\end{prop}

\begin{proof}
To show  \eqref{upper-lower-curvature est}, we first establish the positive lower bound of the principal curvatures. Define 
\begin{eqnarray*}
    \Phi\coloneqq  \Delta h+nh=\sum_{i=1}^n \frac{1}{\k_i}, 
\end{eqnarray*} which is the harmonic  curvature of $\S_t$. 

First of all, we \textbf{claim} that if $\theta\in (0,\pi/2)$, there holds
\begin{eqnarray} 
\label{boundary of Phi} 
    \n_\mu \Phi \leq 0 \quad \text{ on } \p \C_\theta.
\end{eqnarray} 
In fact, along $\p{\mathcal C}_\theta$, we choose  an orthonormal frame $\{e_i\}_{i=1}^n$  such that $e_n=\mu$. Since  $\n_\mu h=\cot\theta h$ on $\p \C_\theta$, then for all $1\leq \beta\leq n-1$,  
		\begin{eqnarray*}
			\n_\beta (h_n)=\cot\theta \n_\beta h,
		\end{eqnarray*} 
  together with the Gauss-Weingarten equation of $\p \C_\theta \subset \C_\theta$,  we have
		\begin{eqnarray}\label{h_1n}
			h_{\beta n}=	\n^2h(e_\beta ,e_n)=\n_{e_\beta}(\n h(e_n))-\n h(\n_{e_\beta}e_n) =0.
		\end{eqnarray}
		Applying the Gauss-Weingarten equation of $\p \C_\theta \subset \C_\theta$ again and combining with \eqref{h_1n}, we have 
		\begin{eqnarray*}
			\n_{\beta\beta n}h&=&(\n^3 h)(e_\beta,e_\beta,e_n)	=(\n_{e_\beta }(\n^2 h))(e_\beta ,e_n)\\&=& \n_{e_\beta}(\n^2 h(e_\beta,e_n))-\n^2 h(\n_{e_\beta }e_\beta,e_n)-\n^2 h(e_\beta,\n_{e_\beta}e_n)
			\\&=&\cot\theta h_{nn}-\cot\theta h_{\beta \beta}.
		\end{eqnarray*} 
		The commutator formula  on $\C_{\theta}$ implies  
		\begin{eqnarray*}
			\n_{n\beta\beta } h=\n_{\beta\beta n}h-h_n=\cot\theta ( h_{nn}-  h_{\beta\beta }- h),
		\end{eqnarray*} then it follows
  \begin{eqnarray}\label{a aan}
     \n_{n} a_{\beta\beta} =\n_{n}h_{\beta\beta} + h_{n}= \cot\theta (a_{nn}-a_{\beta\beta}).
  \end{eqnarray}  
   From \eqref{robin of K}, we have
    \begin{eqnarray*}
   \sum_{\beta=1}^{n-1}     a^{\beta\beta}\n_{n}a_{\beta\beta}+a^{nn}\n_{n}a_{nn}=0,
    \end{eqnarray*}that is, \begin{eqnarray*}
       \n_{n} a_{nn}=-   \sum_{\beta=1}^{n-1}    
 \frac{a^{\beta\beta }}{a^{nn}}\n_{n}a_{\beta\beta},
    \end{eqnarray*}
 which,   together with \eqref{a aan}, implies
\begin{eqnarray*}
    \n_\mu \Phi&=&   \sum_{\beta=1}^{n-1}\n_\mu a_{\beta\beta}+\n_\mu a_{nn}
    \\&=& \cot\theta  \sum_{\beta=1}^{n-1} \left( 1-\frac{ a^{\beta\beta}}{a^{nn}}\right) (a_{nn}-a_{\beta\beta}) \leq 0,
\end{eqnarray*}
where the last inequality follows from the positive definiteness of $(a_{ij})$ and $\theta \in(0,\pi/2)$. Hence, we prove the assertion \eqref{boundary of Phi}.

Next, we compute the evolution equation of $\Phi$, bear in mind that $$\Delta \ell+n\ell =n.$$ From Eq. \eqref{prop-scalar equ},  we derive
\begin{eqnarray*}
    \p_t \Phi &=&\Delta (\p_t h)+n\p_t h
    \\&=& -\ell \Delta K-2\n \ell \cdot \n K-K\Delta \ell   -n\ell K
    \\&=&- \ell K \Delta K-2\n \ell\cdot \n K-nK,
\end{eqnarray*}
and
\begin{eqnarray}\label{eq_add1}
    \n_i K=-Ka^{kl}\n_{i}a_{kl},
\end{eqnarray}
By the covariant derivative commutator formula on $\C_\theta$, we have
\begin{eqnarray*}
\n_{j}\n_{i}h_{pq}=\n_{q}\n_{p}h_{ij}+2(h_{pq}\delta_{ij}-h_{ij}\d_{pq})+h_{iq}\d_{pj}-h_{jp}\d_{iq},
\end{eqnarray*} which implies
\begin{eqnarray*} \n_{j}\n_{i}a_{pq}&=&\n_{j}\n_{i}h_{pq}+h_{ij}\d_{pq}
 \\&=&\n_{q}\n_{p}a_{ij} +a_{pq}\delta_{ij}-a_{ij}\d_{pq}+a_{iq}\d_{pj}-a_{jp}\d_{iq}.
\end{eqnarray*} 
Then
\begin{eqnarray*}
    \n_{ij} K&=& -\n_j K a^{kl}\n_{i}a_{kl}-Ka^{kl}\n_{j}\n_{i}a_{kl}-K\n_{j}a^{kl}\n_{i}a_{kl}
    \\&=&  -Ka^{kl} (\n_{l}\n_{k}a_{ij}+a_{kl}\d_{ij}-a_{ij}\d_{kl}+a_{il}\d_{kj}-a_{jk}\d_{il}) 
 \\&& +Ka^{kp} \n_{j}a_{pq} a^{ql}\n_{i}a_{kl}+ Ka^{pq}\n_{j}a_{pq}a^{kl}\n_{i}a_{kl} 
\end{eqnarray*}
and
\begin{eqnarray*}
    \Delta K&=&  Ka^{kp} \n_{i}a_{pq} a^{ql}\n_{i}a_{kl}+ Ka^{pq}\n_{i}a_{pq}a^{kl}\n_{i}a_{kl} 
    \\&& -Ka^{kl}\n_{l}\n_{k}\Phi-n^2 K+\Phi  H K.
\end{eqnarray*}
Altogether yields
\begin{eqnarray*}
    \p_t\Phi &=&   \ell Ka^{kl}\n_{l}\n_{k}\Phi+ n^2 \ell -\ell K \Phi H -n K 
   -2 \n \ell\cdot \n K  \\&&  -\ell K a^{kp} \n_{i}a_{pq} a^{ql}\n_{i}a_{kl}- \ell  K a^{pq}\n_{i}a_{pq}a^{kl}\n_{i}a_{kl}.
\end{eqnarray*}
From \eqref{boundary of Phi}, we  have that $\Phi$ attains its maximum value  at some interior point, say $\xi_0 \in \C_{\theta} \setminus \partial \C_{\theta}$. Assume that $a_{ij}$ is diagonal at this point, which in turn implies that $a^{ij}$ is also diagonal at $\xi_0$. Then at $\xi_0$, we find
\begin{eqnarray} \notag
    0 &\leq &\L \Phi=\p_t\Phi- \ell Ka^{kl}\n_{k}\n_{l}\Phi   \notag
    \\&=& (n\ell-1) nK - \ell K \Phi H 
    \notag \\&&  -  \ell K a^{kp} \n_{i}a_{pq} a^{ql}\n_{i}a_{kl} - \ell K a^{pq}\n_{i}a_{pq}a^{kl}\n_{i}a_{kl}  -2 \n\ell \cdot \n K.\label{evo of phi}
\end{eqnarray}
From \eqref{eq_add1} we have 
\begin{eqnarray*}
   -2  \n \ell\cdot \n K=2\cos\theta \<e_i,e\> K a^{kk}\n_{i}a_{kk}.
\end{eqnarray*}
By using the Cauchy-Schwarz inequality and the convexity of $\S_t$,  we obtain
\begin{eqnarray}
  -\ell \sum_{k,l,i} a^{kk}a^{ll}(\n_{i}a_{kl})^2-  \ell  \sum_i (\sum_k a^{kk}\n_{i}a_{kk} )^2-2\cos\theta \<e_i,e\> a^{kk} \n_{i}a_{kk} \leq  C  ,~~~~~~\label{Kalpha upper}
\end{eqnarray} 
where the constant $C$ only depends on $n$.
  Inserting \eqref{Kalpha upper} into  \eqref{evo of phi}, and using the fact that $H\geq nK^{1/n}$ , we derive that
\begin{eqnarray*}
     \ell K^{\frac{1}{n}} \Phi\leq \frac{1}{n} \ell H\Phi \le ( n \ell -1)+C,
\end{eqnarray*}
which, together with  Proposition \ref{lower bound of K}, gives 
\begin{eqnarray*}
  \Phi \leq C_{0},
\end{eqnarray*}for some positive constant $C_{0}$ depending on $n$ and $\S_{0}$. It implies  \begin{eqnarray}\label{kappa-lower-bound}
    \k_i\geq C^{-1}_{0}.
\end{eqnarray}
The upper bound of $\k_i$ follows directly from \eqref{kappa-lower-bound} and Proposition \ref{upper bound of K}, since
\begin{eqnarray*}
    \k_i = K \prod_{j\neq i} \k_j^{-1} \leq C C_0^{-(n-1)} \coloneqq C_1.
\end{eqnarray*}
This completes the proof.
\end{proof}

 Now, we are ready to follow the approach presented by Tso in \cite[Section 4]{Tso85} (see also \cite[Chapter 15]{ACGL} or \cite[Section 4]{GHL}) to complete the proof of Theorem \ref{main-existence}. 
\subsection {Proof of Theorem \ref{main-existence}}
\begin{proof} 
Let  $\widehat{\S}_t$ denote the convex domain enclosed in $\ov{\RR^{n+1}_+}$ by $\S_t$ and $\p\ov{\RR^{n+1}_+}$. Since $\S_t$ is the solution of flow \eqref{c-Gauss curvature flow}, 
by the comparison principle   we know that $\widehat{\S}_t\subset \widehat{\S}_s$, for any $s<t$. 
 
First, we demonstrate that the capillary inner radius $\rho_{-}(\wh\S_t ,\theta)$ of $\S_t$ tends to zero as $t\rightarrow T^{\ast}$. In fact, if this is not true, since the flow is a contracting flow and by the definition of $\rho_-(\wh\S_t,\theta)$, we know that there exists a spherical cap with some radius $r_{0}>0$ which is contained in $\Sigma_{t}$ for $ t\in [0, T^{\ast})$. Combining Proposition \ref{curvature est} and
  the standard parabolic equation theory (see, e.g., \cite[Theorem 6.1, Theorem 6.4 and Theorem 6.5]{Dong}, and also \cite[Theorem 14.23]{Lieberman-book-parabolic}), we obtain a uniform $C^{k,\alpha}$ estimate of the solution to flow \eqref{support fun eq. of capillary alpha-GCF} for all $k\geq 2$ and $\alpha\in (0,1)$. Therefore, flow \eqref{support fun eq. of capillary alpha-GCF} can be extended up to time $T^*+\varepsilon$ for some $\varepsilon>0$, which contradicts the fact that $T^*<+\infty$ is the maximal existence time of flow \eqref{support fun eq. of capillary alpha-GCF}.

Next, we prove that as $t\rightarrow T^{\ast}$, the hypersurface $\Sigma_{t}$ converges to a point. Suppose, by contradiction, that $\cap_{t\geq 0}\widehat{\Sigma}_{t}$ is not a point. Let $\text{diam}(\cap_{t\geq 0}\widehat{\Sigma}_{t})=l_{0}=\text{dist}(x_{1}, x_{2})$  where $x_{1}, x_{2}\in \cap_{t\geq 0}\widehat{\Sigma}_{t}$. 
Since the inner capillary radius tends to zero, the volume of $\widehat{\Sigma}_{t}$  also converges to zero. Thus, there exists a two dimension hyperplane $\mathcal{H}$ such that the area of $\bar{H}\coloneqq \mathcal{H}\cap \left(\cap_{t\geq 0}\widehat{\Sigma}_{t}\right)$  converges to zero.  Along the boundary of $\bar{H}\cap \partial (\cap_{t\geq 0}\widehat{\Sigma}_{t})$, there must be some point with arbitrarily small principal curvature. However, Proposition \ref{curvature est} implies the principal curvatures of $\Sigma_{t}$  have a uniformly lower bound. Thus, we derive a contradiction, and  $\widehat{\Sigma}_{t}$ converges to a point as $t\rightarrow T^{\ast}$.
\end{proof}

\section{Normalized capillary Gauss curvature flow}\label{sec4}
The primary objective of this section is to show that the normalized capillary Gauss curvature flow \eqref{GCF-capillary-normalized} below will converge to a soliton. Specifically, we establish the uniform estimates for the solution of the normalized flow and analyze its asymptotic behavior.  

\subsection{Normalized capillary Gauss curvature flow} 
Suppose that $\wt X\colon M\times[0,T^{\ast})\to \ol{\RR^{n+1}_+}$ satisfies flow \eqref{c-Gauss curvature flow}, that is, \begin{eqnarray}\label{GCF-capillary-2}
\left\{
\begin{array}{llll}
\p_\tau \wt X(\cdot, \tau) &=&-\wt K(\cdot, \tau)  \tilde\nu(\cdot,\tau) \quad &
\hbox{ in } M\times [0,T^*),\\
\<\tilde{\nu} ,e \> &=&0
\quad  & \hbox{ on }\partial M\times [0,T^*),\\
X(\cdot,0)&=& X_0(\cdot) & \hbox{ on } M.
\end{array} \right.
\end{eqnarray}
From Theorem \ref{main-existence}, we have proved that $\wt \S_\tau\coloneqq \wt X(M,\tau)$ shrinks to some point $p\in \p \ov{\RR^{n+1}_+}$ as $\tau \to T^*$. Through a horizontal translation, we can assume that the limiting point $p$ is the origin. To investigate the limiting shape of flow \eqref{GCF-capillary-2}, we consider a normalized flow of \eqref{GCF-capillary-2} by ensuring the enclosed domain of $\wh{\wt\S}_0$ has the same volume as $\wh{\C_\theta}$ (see  Proposition \ref{volume-preserving} below). More specifically, we define a  scaling of $\wt \S_\tau$ with respect to the limiting point $p\in \p \ov{\RR^{n+1}_+}$ by
\begin{eqnarray}\label{scaling of X}
    X(x, t)\coloneqq e^t \wt X(x, \tau), ~~~x\in M,
\end{eqnarray}where $ t\coloneqq \frac 1 {n+1} \log \left(\frac{{\rm{Vol}}(\wh{\C_\theta})}{{\rm{Vol}} (\wh {\wt \S}_\tau)}\right)$ and $\tau\in[0, T^*)$.  In view of Theorem \ref{main-existence}, it is obvious that $X(M,t)$ is well-defined for all $t\in[0, +\infty)$. 

By a direct computation, $X(x,t)$ satisfies
\begin{eqnarray*}
    \p_t X(x,t)&= & X(x,t)+e^t \p_\tau \wt X(x,\tau) \frac{\p \tau}{\p t}
    \\&=&X(x,t)-e^t  \frac{(n+1){\rm{Vol}}(\wh{\wt \S}_\tau)}{\int_{\widetilde{\S}_{\tau} }\wt K(\cdot, \tau) (1+\cos\theta\<\nu,e\>) d\mu_\tau} \wt K \tilde \nu
    \\&=& X(x,t) - K(x,t)\tilde\nu(x,t).
\end{eqnarray*}
where $d\mu_{\tau}$ is the area element of $\widetilde{\S}_{\tau}$.

In summary, after the scaling, $X(x, t)\colon  M\times[0,\infty)\to \ol{\RR^{n+1}_+}$ given by \eqref{scaling of X} satisfies 
\begin{eqnarray}\label{GCF-capillary-normalized}
\left\{
\begin{array}{llll}
\p_t X(\cdot, t) &=&X(\cdot,t)- K(\cdot, t)\tilde\nu(\cdot,t)  \quad &
\hbox{ in } M\times [0,+\infty),\\
\<\tilde\nu ,e\> &=&0
\quad  & \hbox{ on }\partial M\times [0,+\infty),\\
X(\cdot,0)&=& X_0(\cdot) & \hbox{ on } M.
\end{array} \right.
\end{eqnarray}

\begin{prop}\label{volume-preserving}    
If ${\rm{Vol}}(\wh{\S}_0)={\rm{Vol}}({\wh{\C_\theta}})$, then along flow \eqref{GCF-capillary-normalized}, the volume of $\wh\S_t $ is preserved for all $t\geq 0$, 
\end{prop}
\begin{proof}
    From \cite[Theorem 1.1 for $k=0$]{WWX22},   we have
    \begin{eqnarray*}
        \frac{d}{dt}{\rm{Vol}}(\widehat {\S}_{t})&=&\int_{\S_{t}}hd\mu_{t} - \int_{\S_t} K(\cdot,t ) (1+\cos\theta\<\nu, e\>)d\mu_{t}\\
        &=&(n+1) \left( {\rm{Vol}}(\wh\S_t )-{\rm{Vol}}(\widehat{\C_{\theta}})\right),
    \end{eqnarray*}
where $d\mu_{t}$ is the area element of $\S_{t}$. Together with the assumption that ${\rm{Vol}}(\wh\S_0 )={\rm{Vol}}(\widehat{\C_{\theta}})$, we derive that  ${\rm{Vol}}(\wh\S_t)={\rm{Vol}}(\wh{\C_\theta})$  for all $t\geq 0$. Therefore, the conclusion follows.
   \end{proof} 

 Similarly, as in the previous section, the normalized capillary Gauss curvature flow \eqref{GCF-capillary-normalized} can be reformulated as a scalar parabolic flow of Monge–Amp\`ere type equation with Robin boundary value condition to the function $h(\xi,t)$, which is the support function of $\S_t\coloneq X(M,t)$ with respect to the origin.
 \begin{prop}
     Let $\S_t\coloneqq X(M,t)$ be the solution of flow \eqref{GCF-capillary-normalized}. Then, $h$ satisfies
     \begin{eqnarray}\label{Support-Normalized} \left\{
\begin{array}{llll}
  \p_t h(\xi, t)  &=&h(\xi, t)-\ell (\xi)\left[\det(\n^2h+hI)\right]^{-1}(\xi, t) & \text{ in }~ \C_\theta\times[0,+\infty),\\ 
	\n_\mu h&=& \cot\theta h & \text{ on } ~\p \C_\theta\times[0, +\infty), \\
 h(\cdot,0) &=& h_0(\cdot)  & \text{ on } ~ \C_\theta, \end{array}  \right.
\end{eqnarray} where  $\mu$ is the unit outward normal of $\p \C_\theta\subset  \C_\theta$ and $h_0$ is the support function of $\S_0$.
 \end{prop} 
  
\begin{proof}
    The proof is similar to Proposition \ref{prop-scalar equ}, thus we omit it here for the conciseness.
\end{proof}

Equivalently, from \eqref{Support-Normalized}, the capillary support function $u(\xi,t)$ of $\S_t$ (with respect to the origin) satisfies
\begin{eqnarray}\label{eq-scalar-capillary-support-normalized-GCF} \left\{
\begin{array}{llll}
  \p_t u&=&u-\frac 1 {\det \left(\n^2 (\ell u)+\ell u\s \right)}   & \text{ in } ~\C_\theta\times[0,+\infty),\\ 
	\n_\mu u&=& 0 & \text{ on }~ \p \C_\theta \times[0, +\infty), \\
 u(\cdot,0) &=& u_0(\cdot)  & \text{ on } ~\C_\theta, \end{array}  \right.
\end{eqnarray}
where $u_0$ is the capillary support function of $\S_{0}$.

Next, we investigate the asymptotic behavior of flow \eqref{GCF-capillary-normalized} and prove the main result in this paper, the second part of  Theorem \ref{main-thm}. For readers' convenience, we restate it in the following.

\begin{thm}\label{GCF-normal to soliton}
Let $\S_0$ be a strictly convex capillary hypersurface in $\ol{\R^{n+1}_+}$ and $\theta\in(0,\pi/2)$, then flow \eqref{GCF-capillary-normalized} has a smooth solution $\S_t\coloneqq X(\cdot,t)$, which is strictly convex capillary hypersurface for all $t\in[0,+\infty)$. Moreover,  the solution of flow \eqref{GCF-capillary-normalized} converges to a smooth, strictly convex capillary hypersurface $\S_\infty\subset \ol{\RR^{n+1}_+}$ as  $t\to +\infty$. 
Furthermore, $\S_\infty$ satisfies 
\begin{eqnarray}\label{eq-soliton-normalized-GCF}
\left\{
\begin{array}{llll}
	K &=& \frac{\<X,\nu\> } {1+\cos\theta\<\nu, e\>} & \text{ in } \S_\infty,\\ 
	\<\nu,e\>&=& -\cos\theta & \text{ on } \p \S_\infty.\end{array}
\right.
\end{eqnarray}
\end{thm}

\ 

The proof of Theorem~\ref{GCF-normal to soliton} is inspired by  \cite{GN} and \cite{AGN}. A key step is the introduction of a suitable geometric monotone quantity along flow~\eqref{GCF-capillary-normalized}, the \textit{capillary entropy functional} $\mathcal{E}_{\theta}(\wh\S)$, which enables us to establish a positive lower bound for the capillary support function. Before proceeding with the proof of Theorem~\ref{GCF-normal to soliton}, we first introduce the notion of capillary entropy and establish several useful properties in the following subsection.

\subsection{Capillary entropy and related properties}\label{sec-4.2}
In this subsection, we introduce the capillary entropy functional for capillary convex bodies. This functional plays an important role in investigating the asymptotic shape of the normalized capillary Gauss curvature flow \eqref{GCF-capillary-normalized}.  The concept of this entropy is inspired by the one defined for closed convex hypersurfaces, as presented in the papers of Guan-Ni \cite{GN} and Firey \cite{Firey}.
 
\begin{defn}[Capillary entropy functional]\label{def-cap-entropy}
Let $\wh\S\in \K_\theta$ be a capillary convex body. The capillary entropy functional $\E_{\theta}$ associated with $\wh\S$ is defined by 
\begin{eqnarray*}
    \mathcal{E}_{\theta}(\wh\S)\coloneqq \sup_{z\in {\rm{int}}(\widehat{\partial\S})}    \E_{\theta}(\wh\S,z),
\end{eqnarray*} where
\begin{eqnarray*}
    \E_{\theta}(\wh\S,z)\coloneqq 
        \frac 1 {\omega_\theta}\int_{\C_\theta} \left(\log u_{z}(\xi)\right) \ell(\xi)  d\s,
\end{eqnarray*}
and  $\omega_\theta\coloneqq |\C_\theta|-\cos\theta |\wh{\p\C_\theta}|$ is the capillary area of $\C_\theta$. 
 
\end{defn}

In the following, we present some properties of the capillary entropy functional. Similar properties for closed convex hypersurfaces, one can refer to \cite[Section~2]{GN}, \cite[Section~2]{AGN}, and \cite[Section 2, 3]{BG}.
The underlying ideas of the proofs of some propositions are essentially the same as those in \cite{GN}. Nevertheless, for the sake of completeness, we provide the full arguments here, including the necessary modifications to account for the capillary boundary. We also remark that these results are of independent interest and may be useful for other related problems for capillary hypersurfaces.

To begin with, by using the capillary  Blaschke-Santal\'o inequality established in Proposition \ref{prop-santalo-ineq}, we show that the capillary entropy functional $\E_{\theta}(\widehat{\S})$ is bounded from below. In particular, if Conjecture \ref{conj-sharp-capillary-B-S-ineq}  holds, this would imply that the capillary entropy functional is non-negative. Nevertheless, \eqref{lower bdd} is sufficient for our purposes later on.

\begin{prop}\label{lem-low-entroy}
 Let $\wh\S\in \K_\theta$ with  $\theta\in(0, \pi/2)$ and  ${\rm{Vol}}(\widehat \S) ={\rm{Vol}}(\widehat{\C_{\theta}})$. If $z_s \in \widehat{\partial \S} $
    is the capillary Santal\'o point of $\widehat{\S}$, 
   then
    \begin{eqnarray}\label{lower bdd}
        \frac 1 {\omega_\theta}\int_{\C_\theta} \left(\log u_{z_s}(\xi)\right) \ell(\xi) d\s \ge -C,
        \end{eqnarray}
   where the constant $C > 0$ only depends on $n$ and $\theta$. 
\end{prop}

\begin{proof}
    The proof follows closely the idea of Guan-Ni \cite[Proposition~1.1]{GN}, by adjusting to use Proposition \ref{prop-volume-polar-body} and Proposition \ref{prop-santalo-ineq}. 

    From  Proposition \ref{prop-volume-polar-body}, 
\begin{eqnarray*}
{\rm{Vol}}(\widehat{\S}^*_{z_s}) 
&=&\frac{1}{n+1}\int_{\C_{\theta}}\frac{\ell(\xi)}{u^{n+1}_{z_{s}}(\xi)} d\s.
 \end{eqnarray*}
By applying Jensen's inequality, we deduce that
\begin{eqnarray}\label{Jen-ine}
 \exp \left(   \frac{1}{\omega_{\theta}}\int_{\C_{\theta}}\log \left(\frac{1}{u^{n+1}_{z_{s}}(\xi)} \right)\ell(\xi) d\s\right)\leq \frac{1}{\omega_{\theta}}\int_{\C_{\theta}}\frac{1}{u^{n+1}_{z_{s}}(\xi)}\ell(\xi) d\s=\frac{n+1}{\o_\theta} {\rm{Vol}}(\widehat{\S}_{z_s}^{\ast}).
\end{eqnarray}
Inserting the simple fact that $\omega_{\theta}=(n+1){\rm{Vol}}(\widehat{\C_{\theta}})$ into \eqref{Jen-ine}, and together with Proposition \ref{prop-santalo-ineq},  we obtain
\begin{eqnarray*}
   \exp \left(-\frac{(n+1)}{\omega_{\theta}}\int_{\C_{\theta}} \left(\log u_{z_{s}}(\xi)\right)\ell(\xi) d\s  \right)
   \leq \frac{{\rm{Vol}}(\widehat{\S}_{z_s}^{\ast})}{{\rm{Vol}}(\wh\C_\theta)}\leq \frac{1}{2}\left(\frac{{\rm{Vol}}(\mathbb{B}^{n+1})}{{\rm{Vol}(\wh\C_{\theta})}}\right)^{2},
\end{eqnarray*}
then this implies \eqref{lower bdd}.
\end{proof}

The following proposition shows that the capillary entropy functional $\E_{\theta}(\widehat{\S})$ is attained at a point $z_e$ on the flat part of the boundary of $\widehat{\S}$, namely $\widehat{\partial\S}$. We refer to this point $z_e=z_{e}(\wh\S)$ as the \textit{capillary entropy point} of the capillary convex body $\widehat{\S}$.
We note that Guan-Ni \cite[Lemma 2.2]{GN} proved that the entropy point of a smooth, closed, convex body lies in the interior. In contrast, in our setting, the capillary entropy point lies on the flat part of the boundary of the capillary convex body $\widehat{\S}$.

\begin{prop}\label{prop-exist-cap-entropy-point}
 Let $\wh\S\in \K_\theta$ with  $\theta\in(0, \pi/2)$. Then there exists a unique point $z_{e}\in \widehat{\partial \S} $ such that $$\mathcal{E}_{\theta}(\widehat{\S})=\frac{1}{\omega_{\theta}}\int_{\C_{\theta}}
\left(\log u_{z_{e}}(\xi)\right)\ell(\xi) d\s.$$
\end{prop}

\begin{proof}
We divide the proof into two steps: first, we establish the existence of $z_e$, second, we prove its uniqueness.

\

 \textit{Step 1: Existence}. By Definition \ref{def-cap-entropy}, we know there exists a  sequence $\{z_j\}_{j\geq 1}\subset \text{int}(\wh{\p\S})$ and some point $z_e\in \wh{\p\S}$  such that 
\begin{eqnarray}\label{maximizing-sequence-entropy}
    \mathcal{E}_{\theta}(\widehat{\S}, z_j)\rightarrow \mathcal{E}_{\theta}(\widehat{\S})~ ~~{\rm{and}}~z_j\rightarrow z_{e}~~{\rm{as}}~j\rightarrow +\infty.
\end{eqnarray}
From \eqref{capillary support}, \eqref{defn-capillary-outer-radius} and \cite[Remark 2.8]{MWW}, we know  
\begin{eqnarray}\label{double-control}
    u_{z_j}(\xi)\leq 2\rho_{+}(\widehat{\S}, \theta), ~~\text{ for all } \xi\in \C_\theta, \end{eqnarray}
by Fatou's Lemma and \eqref{maximizing-sequence-entropy}, we deduce 
\begin{eqnarray*}
 \int_{\C_{\theta}}\log \left(\frac{2\rho_{+}(\widehat{\S}, \theta)}{u_{z_{e}}(\xi)}\right)\ell (\xi) d\s\leq \liminf\limits_{j\rightarrow+\infty} \int_{\C_{\theta}}\log \left(\frac{2\rho_{+}(\widehat{\S
    }, \theta)}{u_{z_j}(\xi)}\right)\ell(\xi)d\s,
\end{eqnarray*}
in view of \eqref{maximizing-sequence-entropy}, then it follows
\begin{eqnarray}\label{fatou-lemma-estimate}
    \frac{1}{\omega_{\theta}}\int_{\C_{\theta}} \left(\log u_{z_{e}}(\xi)\right)\ell(\xi)d\sigma\geq \limsup\limits_{j\rightarrow +\infty} \frac{1}{\omega_{\theta}}\int_{\C_{\theta}}(\log u_{z_j}(\xi))\ell(\xi)d\sigma=\mathcal{E}_{\theta}(\widehat{\S}).
\end{eqnarray}
On the other hand, by Definition \ref{def-cap-entropy}, it is obvious that
$$ \frac{1}{\omega_{\theta}}\int_{\C_{\theta}}\left(\log u_{z_{e}}(\xi)\right)\ell(\xi)d\sigma\leq \mathcal{E}_{\theta}(\widehat{\S}),$$
together with \eqref{fatou-lemma-estimate}, we obtain that $z_e$ is the desired supremum point of $\E_\theta(\wh\S)$.

\ 

 \textit{Step 2: Uniqueness}. Without loss of generality, we assume that the origin lies in ${\rm{int}}(\widehat{\partial \S})$. For all point $z\coloneqq (z_{1}, \cdots, z_{n}, 0)\in \widehat{\partial \S}$, it follows from \eqref{uz} that we can rewrite
    \begin{eqnarray*}
        \frac{1}{\omega_{\theta}}\int_{\C_{\theta}}(\log u_{z}(\xi) )\ell(\xi) d\sigma 
        &= &\frac{1}{\omega_{\theta}}\int_{\C_{\theta}}\log \left(u(\xi)-\ell^{-1}(\xi)\sum\limits_{i=1}^{n}z_{i}\xi_{i}\right)\ell(\xi) d\s \\&\eqqcolon & F(z),
    \end{eqnarray*}where $u(\xi)$ is the capillary support function of $\wh\S$ with respect to the origin.
A direct calculation yields
\begin{eqnarray}\label{eq-concavity-F(z)}
    \frac{\partial^{2}F(z)}{\partial z_{i}\partial z_{j}}=-\frac{1}{\omega_{\theta}}\int_{\C_{\theta}}\frac{\xi_{i}\xi_{j}}{\ell(\xi)u_{z}^{2}(\xi)}d\s,~~~~\text{for all}~~ 1\leq i,j\leq n.
\end{eqnarray}
Hence, this implies that $F(z)$ is strictly concave with respect to $z\in \wh{\p\S}$, which in turn yields the uniqueness of $z_e$.
 
\end{proof}

The following proposition asserts that the capillary entropy point $z_{e}$, as established in Proposition~\ref{prop-exist-cap-entropy-point}, lies in the interior of $\widehat{\partial \S}$. Consequently, we have $u_{z_e} > 0$ on $\C_{\theta}$.

\begin{prop}\label{lem entropy}
 Let $\wh\S\in \K_\theta$ with  $\theta\in(0, \pi/2)$. 
Then $\mathcal{E}_{\theta}(\widehat{\S})$ is attained by a unique capillary support function $u_{z_e}$ with respect to the point $z_{e}\in \rm{int}(\widehat{\partial \S})$. Moreover, $u_{z_e}$ satisfies
\begin{eqnarray}\label{ortho}
    \int_{\C_{\theta}}\frac { \xi_i}{u_{z_e}(\xi)} d\s=0,\quad \forall~ i=1,\cdots, n.
    \end{eqnarray}
    Moreover, for any capillary support function $U\neq u_{z_e}$ in $\C_\theta$, there holds
    \begin{eqnarray}\label{ineq-entropy-unique-support-function}
        \E_\theta(\wh\S) > \frac{1}{\o_\theta} \int_{\C_\theta}  \left(\log U (\xi) \right) \ell d\s.
    \end{eqnarray}
    \end{prop}
    
\begin{proof}
We begin by showing that the capillary entropy point $z_{e}$ lies in ${\rm{int}}(\widehat{\partial \S})$. Suppose, to the contrary, that $z_{e} \in \partial \S$. Without loss of generality, we may assume that $z_{e}$ is the origin.  Following the proof and notations introduced in Proposition~\ref{pro-santalo}, for all $t \in (0, t_0)$, the point $z_{t} \coloneqq -t \zeta$ lies in ${\rm{int}}(\widehat{\partial \S})$ and satisfies $u_{z_t}(\xi) > 0$. By \eqref{big}, we obtain
\begin{eqnarray*}
   0&\geq & \frac{d}{dt}\left(\int_{\C_{\theta}} \left(\log u_{z_t}(\xi)\right)\ell (\xi)d\s\right)\Big|_{t=0}=\int_{\C_{\theta}}\frac{\xi_{1}}{u_{z_e}(\xi)}d\s\\
   &=&\int_{\{\xi_{1}>0\}}
    \frac{ \xi_{1}}{u_{z_e}(\xi)}d\sigma+\int_{\{\xi_{1}<0\}}\frac{ \xi_{1}}{u_{z_e}(\xi)}d\s\\
    &=&\int_{\{\xi_{1}>0\}}\left(\frac{ \xi_{1}}{u_{z_e}(\xi)}-\frac{ \xi_{1} }{u_{z_e}(N(\xi))}\right)d\s>0,
\end{eqnarray*}
which leads to a contradiction. Hence we derive that  $z_{e}\in {\rm{int}}(\widehat{\partial \S})$ and in turn $u_{z_e}(\xi)>0$ for all $\xi\in \C_\theta$.  

Next, we verify that the assertion \eqref{ortho} holds. For all $1 \leq i \leq n$, denote $z_e^{r,i} \coloneqq z_{e} - r E_{i}$. For sufficiently small $r_0 > 0$, we have 
$$z_e^{r,i} \in {\rm{int}}(\widehat{\partial\S}) ~~~~\text{ for all } |r| \leq r_0 ~~\text{ and }~~ 1\leq i\leq n.$$
We reformulate
    \begin{eqnarray*}
   \frac{1}{\omega_{\theta}}\int_{\C_{\theta}}\left(\log u_{z_{e}^{r,i}}(\xi) \right)\ell(\xi) d\s=\frac{1}{\omega_{\theta}}\int_{\C_{\theta}}\log \left(u_{z_{e}}(\xi)+\ell^{-1}(\xi)r \xi_{i}\right)\ell(\xi) d\s\eqqcolon G_i(r).
    \end{eqnarray*}
  Proposition \ref{prop-exist-cap-entropy-point} implies that $G_i(r)$ attains its maximum value at $r=0$, hence, 
  \begin{eqnarray*}
    0=\frac{d}{dr}G_i\Big|_{r=0}=\frac{1}{\omega_{\theta}}\int_{\C_{\theta}}\frac{\xi_{i}}{u_{z_{e}}(\xi)}d\s.
    \end{eqnarray*}
    
The last assertion \eqref{ineq-entropy-unique-support-function} follows from the strict concavity of  $F(z)$ in the proof of Proposition \ref{prop-exist-cap-entropy-point}, see \eqref{eq-concavity-F(z)}.  Hence, we complete the proof.

\end{proof}

Next, we show that the capillary entropy functional provides control over the geometric quantities of a capillary convex body. More precisely, the capillary entropy $\mathcal{E}_{\theta}(\widehat{\S})$ controls both the capillary outer and inner radii. This proposition will play a crucial role in establishing the $C^{0}$ estimate for solutions to the normalized capillary Gauss curvature flow \eqref{GCF-capillary-normalized}, ensuring that the solution does not collapse into a lower-dimensional convex set.

 \begin{prop}\label{radius}
Let $\widehat{\S} \in \K_\theta$ with $\theta\in (0, \pi /2)$. Then there exists a positive constant $C$, depending only on $n$ and $\theta$, such that
     \begin{eqnarray}\label{capillary outer}
           \rho_{+}(\widehat{\S}, \theta)\leq C e^{\mathcal{E}_{\theta}(\widehat\S)},
     \end{eqnarray}
     and 
     \begin{eqnarray}\label{capillary inner}
         \rho_{-}(\widehat{\S}, \theta)\geq C^{-1}{\rm{Vol}}(\widehat{\S}) e^{-n \mathcal{E}_{\theta}(\widehat \S)}.
     \end{eqnarray}
 \end{prop} 
 \begin{proof}
Let $\wh{\C_{\rho_{0}, \theta}(x_{0})}$ be the smallest spherical cap containing $\widehat{\S}$, with radius $\rho_{0} \coloneqq \rho_{+}(\widehat{\S}, \theta)$. Let $z \in \S \cap \C_{\rho_{0}, \theta}(x_{0})$. Without loss of generality, assume $x_{0}$ is the origin, and denote $z\coloneqq (z_{1}, 0, \cdots, 0, z_{n+1})$ with $z_{1}>0$ and $ z_{n+1}\geq 0$. Moreover, we have  \begin{eqnarray}\label{z-norm}
      |z|\geq (1-\cos\theta )\rho_{+}(\widehat{\S}, \theta).
  \end{eqnarray}
We begin by proving the assertion \eqref{capillary outer}. The proof is divided into two cases.

\

\noindent {\it Case 1}. $z_{1}>\frac{1}{2}|z|$. Denote $\wh{z}\coloneqq(z_{1}, 0, \cdots, 0, 0)$. Since $\theta\in (0, \pi /2)$, we have
 line segment $$
   L_{1}\coloneqq  \left\{s\widehat{z} \mid 0\leq s\leq 1 \right\}$$ lies inside $\widehat{\p\S}$, this implies that  the support function of the  line segment $L_{1}$ with respect to the point $\frac{\widehat{z}}{2}$ is bounded above by $h_{\frac{\widehat z}{2}}$, i.e., 
\begin{eqnarray*}
      \frac{1}{2}|\<\xi, \widehat{z}\>|=\frac{1}{2}|\<\xi-\cos\theta e, \widehat{z}\>|\leq h_{\frac{\widehat z}{2}}(\xi), \quad \forall~\xi\in \C_{\theta}.
    \end{eqnarray*}
It yields
    \begin{eqnarray}
   &&\log z_{1} \int_{\C_{\theta}}\ell(\xi) d\s-\int_{\C_{\theta}}(\log 2\ell(\xi)) \ell(\xi) d\s+\int_{\C_{\theta}}\log |\xi_{1}|\ell(\xi) d\s\notag  \\
   &=&  \int_{\C_{\theta}}\log\left(\frac{1}{2\ell}|\<\xi, \widehat z\>|\right)\ell(\xi) d\s
   \leq \int_{\C_{\theta}} \left(\log u_{\frac{\widehat z}{2}}(\xi)\right)
 \ell(\xi) d\s \leq \omega_{\theta}\mathcal{E}_{\theta}(\widehat{\S}). ~~~~~~~~\label{case-1}  \end{eqnarray}
 From \eqref{z-norm}, we have
 \begin{eqnarray*}
     \rho_+(\wh\S,\theta) \leq \frac{2 }{1-\cos\theta} z_1,
 \end{eqnarray*}
together with \eqref{case-1}, we conclude  that
\begin{eqnarray*}
    \rho_{+}(\widehat\S, \theta)\leq C_1 e^{\mathcal{E}_{\theta}(\widehat{\S})},
\end{eqnarray*}
where $C_1$ is a positive constant depending only on $n$ and $\theta$.

\

\noindent {\it Case 2}. $z_{1}\leq \frac{1}{2}|z|$. In this case, we have \begin{eqnarray}\label{big-1}
    z_{n+1}\geq \frac{\sqrt{3}}{2}|z|,
\end{eqnarray}
and the line segment 
$$L_{2}\coloneqq  \left\{s\widehat{z}+(1-s)z\mid  0\leq s\leq 1\right \}$$
 lies inside $\widehat\S$. Similarly, the support function of the line segment $L_{2}$ with respect to the point $\widehat{z}$ is bounded above by the support function $h_{\widehat{z}}$ of $\widehat{\S}$ with respect to $\wh z$, i.e., 
\begin{eqnarray*}
z_{n+1}(\xi_{n+1}+\cos\theta) \leq h_{\widehat z}(\xi),\quad \forall~\xi\in\C_{\theta}.
\end{eqnarray*}
Then we get
\begin{eqnarray*}
   &&\log z_{n+1}\int_{\C_{\theta}}\ell(\xi) d\s +\int_{\C_{\theta}}\log (\xi_{n+1}+\cos\theta)\ell(\xi) d\s-\int_{\C_{\theta}} \log \ell(\xi) \ell(\xi) d\s\notag  \\
   &&\leq \int_{\C_{\theta}} \left(\log u_{\widehat{z}}(\xi)\right)\ell(\xi) d\s\leq \omega_{\theta}\mathcal{E}_{\theta}(\widehat{\S}), 
\end{eqnarray*}
which implies 
\begin{eqnarray}\label{zn+1 upper}
    z_{n+1}\leq C_2 e^{\E_{\theta}(\widehat{\S})}.
\end{eqnarray}
where $C_2$ is a positive constant depending only on $n$ and $\theta$.
Combining with \eqref{z-norm}, \eqref{big-1} and \eqref{zn+1 upper}, we also obtain \eqref{capillary outer}. Hence, we complete the proof of \eqref{capillary outer}.

Next, we proceed to show \eqref{capillary inner} with the help of \eqref{capillary outer}.  Recall that in Proposition \ref{prop-santalo-ineq}, we obtain a symmetric convex body $\O$ by reflecting $\widehat{\S}$ with respect to the hyperplane $\{x_{n+1}=0\}$, and $\O$ satisfies
\begin{eqnarray}\label{double}
    2{\rm{Vol}}(\widehat{\S})={\rm{Vol}}(\Omega). 
\end{eqnarray}
Together with \eqref{inner radius} and \eqref{outer radius}, we have
\begin{eqnarray}\label{volume-control}
{\rm{Vol}}(\O)&\leq& 2n \omega_{n-1} \left(\rho_{+}(\widehat{\S})\right)^{n}\rho_{-}(\widehat{\S})\notag \\
&\leq& 2n\omega_{n-1}\left(\sin \theta \rho_{+}(\widehat{\S}, \theta) \right)^{n}\sin\theta \rho_{-}(\widehat{\S}, \theta),
\end{eqnarray}
where $\omega_{n-1}$ is the area of $\mathbb{S}^{n-1}$. 
Inserting \eqref{capillary outer} and \eqref{double} into \eqref{volume-control}, we conclude that \eqref{capillary inner} holds.

\end{proof}

The following proposition shows that the difference between $\mathcal{E}_{\theta}(\widehat{\S})$ and $\mathcal{E}_{\theta}(\widehat{\S}, z)$  provides a quantitative estimate over the distance between $z$ and $z_{e}(\wh\S)$. This estimate will be instrumental in proving that, as $t \to +\infty$, the capillary entropy point of the solution to the flow \eqref{GCF-capillary-normalized} converges to the origin.


\begin{prop}
 Let $\wh\S\in \K_\theta$ with  $\theta\in(0, \pi/2)$  . Then there exists a positive constant $D$ depending only on $n$ and $\theta$, such that 
    \begin{eqnarray}\label{stability-estimate-entropy-point}
        \frac{1}{\omega_{\theta}}\int_{\mathcal{C}_{\theta}}\left(\log u_{z}(\xi)\right)\ell(\xi) d\s \leq \mathcal{E}_{\theta}(\widehat{\S})-De^{-2\mathcal{E}_{\theta}(\widehat{\S})}|z-z_{e}|^{2},~\forall z\in {\rm{int}}(\widehat{\partial\Sigma}),
    \end{eqnarray}
  where $z_{e}$ is the capillary entropy point of $\wh\S$, as established in Proposition \ref{prop-exist-cap-entropy-point}. 
\end{prop}
  \begin{proof}
Proposition \ref{lem entropy} ensures that $z_{e}\in {\rm{int}}(\widehat{\partial\S})$. 
For all $z\in {\rm{int}}(\widehat{\partial\S})$,  the strict convexity of $\widehat{\partial\S}$ implies the line segment 
  $$N\coloneqq  \left\{s z_{e}+(1-s)z \mid  0\leq s\leq 1 \right\} $$  lies in ${\rm{int}}({\widehat{\partial\S}})$. 
 Denote $t=z-z_{e}\eqqcolon (t_{1}, \cdots, t_{n}, 0)$. By \eqref{capillary support}, we rewrite
  \begin{eqnarray*}
      \frac{1}{\omega_{\theta}}\int_{\C_{\theta}}\left(\log u_{z}(\xi)\right)\ell(\xi) d\s
      =\frac{1}{\omega_{\theta}}\int_{\C_{\theta}}\log \left(u_{z_{e}}(\xi)-\ell^{-1} \sum_{i=1}^n t_{i}\xi_{i}\right)\ell(\xi) d\sigma\eqqcolon Q(t).
  \end{eqnarray*}
  By direct calculations, together with \eqref{ortho}, it follows
  \begin{eqnarray*}
      \frac{\partial Q(t)}{\partial t_{i}}\Big|_{t=0}=-\frac{1}{\omega_{\theta}}\int_{\C_{\theta}}\frac{\xi_{i}}{u_{z_e}(\xi)}d\s=0,~~ \text{ for all } \quad 1\leq i\leq n,
  \end{eqnarray*}
and  
\begin{eqnarray*}
    \frac{\partial^{2}Q(t)}{\partial t_{i} \p t_{j}}=-\frac{1}{\omega_{\theta}}\int_{\C_{\theta}}\frac{\xi_{i}\xi_{j}}{\ell(\xi) u_{z}^{2}(\xi)}d\s, ~~\text{ for all } \quad 1\leq i,j \leq n, 
\end{eqnarray*} 
together with Taylor's theorem, it follows that there exists a point $z_*\in{\rm{int}}(\widehat{\partial\S})$, such that
\begin{eqnarray}\label{Qt}
    Q(t)=Q(0)-\frac{1}{2\omega_{\theta}}\int_{\C_{\theta}}\frac{\<\xi, z-z_{e}\>^{2}}{\ell u_{z_*}^{2}(\xi)}d\s.
\end{eqnarray}
Using \eqref{double-control} and \eqref{capillary outer}, we get
\begin{eqnarray}\label{support-lower}
    \frac{1}{u_{z_*}(\xi)}\geq \frac{1}{2\rho_{+}(\widehat{\S}, \theta)}\geq \frac{1}{2C}e^{-\mathcal{E}_{\theta}(\widehat{\S})}.
\end{eqnarray}
Inserting \eqref{support-lower} into \eqref{Qt}, we obtain 
\begin{eqnarray*}
    Q(t)&\leq &Q(0)-\frac{1}{4C^{2}}e^{-2\E_{\theta}(\widehat{\S})}|z-z_{e}|^{2}\int_{\C_{\theta}}\left<\xi, \frac{z-z_{e}}{|z-z_{e}|}\right>^{2}\ell^{-1}(\xi)d\s \\
    &\leq& \mathcal{E}_{\theta}(\widehat{\S})-De^{-2\mathcal{E}_{\theta}(\widehat{\S})}|z-z_{e}|^{2},
    \end{eqnarray*}
where the constant $D$ only depends on $n$ and $\theta$. Hence, we complete the proof.
 \end{proof}

In view of Proposition \ref{prop-exist-cap-entropy-point}, we remark that \eqref{stability-estimate-entropy-point} is equivalent to 
  \begin{eqnarray}\label{ineq-stability-estimate}
      |z-z_e|^2 \leq \frac 1 D e^{2\E_\theta(\wh\S)}\left( \E_\theta(\wh\S,z_e) - \E_\theta(\wh\S, z) \right),~~~ \forall z\in {\rm{int}}(\widehat{\p\S}),
  \end{eqnarray} which provides a stability-type estimate for $\E_\theta(\wh\S, z)$ in terms of the point $z$.


\subsection{Monotonicity and bounds of support function}
In this subsection, we show the key monotone property of the capillary entropy along normalized flow \eqref{eq-scalar-capillary-support-normalized-GCF}. Subsequently, by integrating the properties of the capillary entropy derived in Section  \ref{sec-4.2}, we can establish the positive lower bound of the capillary support function.  For notational convenience, along the flow \eqref{GCF-capillary-normalized}, we denote the capillary entropy point $z_{e}(\widehat{\S}_{t})$ of $\E_\theta(\wh\S_t)$ by $e(t)$ for all $t \geq 0$.

\begin{prop}\label{prop-mono}
Along flow \eqref{GCF-capillary-normalized}, the capillary entropy $\mathcal{E}_{\theta}(\wh\S_t )$ is non-increasing for all $t\geq 0$.  
\end{prop}

\begin{proof}
    For any $t_{0}>0$, 
   let $u_{e(t_{0})}$ be the capillary support function of $\wh\S_{t_0} $ with respect to the capillary entropy point $e({t_{0}})\in  \text{int}(\widehat{\partial\S}_{t_{0}})$. 
   We define the function
   $$U(\xi,t)\coloneqq u(\xi,t)-e^{t-t_{0}}\frac{\<e(t_{0}), \xi\>}{\ell(\xi)}, ~~~\xi\in \C_\theta, ~~t\geq 0.$$
 By the very definition, we note that \begin{eqnarray}\label{eq-U-t-0}
  U(\xi,t_0)=u_{e(t_0)}(\xi), \text{ for all } \xi\in \C_\theta.
  \end{eqnarray} Hence $U(\cdot,t)>0$ in $\C_\theta$ for all $t\in(t_0-\varepsilon, t_0]$, where $\varepsilon>0$ is  some sufficiently small constant.
   Moreover, it is readily seen that $U(\cdot, t)$ satisfies $\n^2(\ell U)+(\ell U) I>0$ in $\C_\theta$ and $\n_\mu (\ell U)=\cot\theta \ell U$ on $\p\C_\theta$, then $U$ is the capillary support function by using \cite[Definition 2.5 and Proposition 2.6]{MWWX}. Thus, by applying \eqref{ineq-entropy-unique-support-function} to $\wh\S_t$, we obtain
   \begin{eqnarray}\label{ineq-entropy-unique-support-fuction-flow}
       \E_\theta(\wh\S_t) \geq  \frac 1 {\o_\theta} \int_{\C_\theta} \left(\log U(\xi,t) \right)\ell d\s, ~~~~~\text{ for all } t\in (t_0-\varepsilon, t_0].
   \end{eqnarray}
   On the other hand, it follows from \eqref{eq-scalar-capillary-support-normalized-GCF} that $U(\cdot,t)$ satisfies
   \begin{eqnarray}\label{equation-U-flow} \left\{
\begin{array}{llll}
  \p_t U &=&U - K(\xi,t)   & \text{ in } ~\C_\theta ,\\ 
	\n_\mu U&=& 0 & \text{ on }~ \p \C_\theta ,\end{array}  \right.
\end{eqnarray}where we recall that $K(\xi,t)=\frac 1 {\det \left(\n^2 h+h \s \right)}=\frac 1 {\det \left(\n^2 (\ell U)+\ell U \s \right)}.$
A direct calculation and using \eqref{equation-U-flow}, it follows     \begin{eqnarray}\label{d-1}
        \frac{d}{dt} \int_{\C_{\theta}}(\log U(\xi,t))\ell(\xi) d\s = \int_{\C_{\theta}}\frac{U(\xi,t)-K(\xi, t)}{U(\xi,t)}\ell(\xi) d\s.
    \end{eqnarray}
Note that
\begin{eqnarray}
    && \int_{\C_{\theta}}\left(\sqrt{\frac{U(\xi,t)}{K(\xi, t)}}-\sqrt{\frac{K(\xi, t)}{U(\xi,t)}}\right)^{2}\ell (\xi)d\s \notag \\
    &=& \int_{\C_{\theta}}\left(\frac{K(\xi, t)}{U(\xi,t)}+\frac{U(\xi,t)}{K(\xi, t)}-2\right)\ell(\xi) d\s \notag \\
    &=& \int_{\C_{\theta}}\frac{K(\xi, t)}{U(\xi,t)}\ell(\xi) d\s+ \int_{\C_\theta}  \frac{h(\xi,t)-e^{t-t_0}\<e(t_0),\xi\>}{K(\xi,t)} d\s-2\int_{\C_{\theta}}\ell(\xi) d\s. \label{eq-square-form-Sigma-t}  
\end{eqnarray}
Next, we analyze the middle integral term. From \cite[Proposition 2.12 and its proof]{MWWX} and Proposition \ref{volume-preserving}, we have
\begin{eqnarray}
    \int_{\C_\theta} \frac {h(\xi,t)}{K(\xi,t)} d\s &= & \int_{\C_\theta}  h \det(\n^2 h+h\s) d\s = (n+1)\text{Vol}(\wh\S_t) \notag\\&=&(n+1)\text{Vol}(\wh\C_\theta)
    =\int_{\C_\theta}\ell d\s. \label{eq-area-formula-Sigma-t}
\end{eqnarray}
By divergence theorem on $\wh\S_t$, there holds
\begin{eqnarray}
\int_{\C_\theta} \frac{\<e(t_0), \xi\>}{K(\xi,t)} d\s &= & \int_{\C_\theta} \frac{\<e(t_0), \xi-\cos\theta e\>}{K(\xi,t)} d\s = \int_{\S_t} \< e(t_0), \nu(\cdot,t)\> d\mu_t  \notag \\&=&\int_{\wh\S_t} \text{div}_{\RR^{n+1}_+} \left( e(t_0)\right) dV=0   
    .\label{eq-divergence-thm-Sigma-t}
\end{eqnarray}
Inserting \eqref{eq-area-formula-Sigma-t} and \eqref{eq-divergence-thm-Sigma-t} into \eqref{eq-square-form-Sigma-t}, we obtain
\begin{eqnarray*}
\int_{\C_{\theta}}\left(\sqrt{\frac{U(\xi,t)}{K(\xi, t)}}-\sqrt{\frac{K(\xi, t)}{U(\xi,t)}}\right)^{2}\ell (\xi)d\s =-\int_{\C_\theta} \left(1-\frac{K(\xi,t)}{U(\xi,t)}  \right) \ell d\s,
\end{eqnarray*}
together with \eqref{d-1}, we derive 
\begin{eqnarray*}
\frac{d}{dt} \int_{\C_{\theta}}(\log U(\xi,t))\ell(\xi) d\s =- \int_{\C_{\theta}}\left(\sqrt{\frac{U(\xi,t)}{K(\xi, t)}}-\sqrt{\frac{K(\xi, t)}{U(\xi,t)}}\right)^{2}\ell (\xi)d\s\leq 0,
\end{eqnarray*}
    which holds for all $t\in(t-\varepsilon, t_0]$. In view of \eqref{eq-U-t-0}, it yields
    \begin{eqnarray*}
        \int_{\C_\theta} \left(\log U(\xi,t) \right)\ell (\xi)d\s \geq \int_{\C_\theta} \left(\log U(\xi,t_0)\right)\ell (\xi)d\s =\o_\theta \E_\theta(\wh\S_{t_0}),
    \end{eqnarray*} together with \eqref{ineq-entropy-unique-support-fuction-flow}, we obtain
    \begin{eqnarray*}
        \mathcal{E}_{\theta}(\wh\S_t )\geq \frac{1}{\omega_{\theta}}\int_{\C_{\theta}}\left(\log U(\xi,t)\right)\ell(\xi)d\sigma\geq \mathcal{E}_{\theta}(\wh\S_{t_0} ),
        \end{eqnarray*}holds for all $t\in(t-\varepsilon, t_0]$. Therefore, the conclusion holds due to the arbitrariness of $t_0>0$.
\end{proof}

From Proposition \ref{lem-low-entroy} and Proposition \ref{prop-mono}, we see that $\E_{\theta}(\widehat\S_{t})$ is uniformly bounded from below and non-increasing. Thus, we know that $\E_{\infty}\coloneqq\lim\limits_{t\rightarrow +\infty}\E_{\theta}(\widehat\S_{t})$ is well-defined. 
\begin{prop}\label{pro-limit}
Suppose $u(\xi, t)$ is the positive solution to flow \eqref{eq-scalar-capillary-support-normalized-GCF}, then for all $t\geq 0$, there holds
    \begin{eqnarray}\label{limit prop}
       \int_{\C_{\theta}}\left(\log u(\xi,t)\right)\ell(\xi) d\s\geq \o_\theta\E_{\infty}+\int_{t}^{+\infty}\int_{\C_{\theta}}\left(\sqrt{\frac{K(\xi, t)}{u(\xi,t)}}-\sqrt{\frac{u(\xi, t)}{K(\xi, t)}}\right)^{2} \ell(\xi) d\s dt.
    \end{eqnarray}
\end{prop}
\begin{proof}
Fix $T_{1}>0$ and choose $T>T_{1}$. Let $a^{T}=(a_{1}^{T}, \cdots, a_{n}^{T}, 0)\eqqcolon e(T) \in \text{int}(\wh{\p\S}_T)$ be the capillary entropy point of $\wh\S_T $. Consider the function
\begin{eqnarray*}
    u^{T}(\xi, t)\coloneqq u(\xi, t)-e^{(t-T)}\ell^{-1} (\xi)\sum\limits_{i=1}^{n}a_{i}^{T}\xi_{i},~~~~\xi\in\C_\theta,~~ t\geq 0.
\end{eqnarray*} It is obvious that  \begin{eqnarray}\label{eq-rewrite-u-T-at-time-T}
    u^T(\xi, T)=u_{e(T)}(\xi,T),~~\forall \xi\in \C_\theta,
\end{eqnarray} is exactly the capillary support function of $\wh\S_T$ with respect to $e(T)$. Moreover, it follows from \eqref{eq-scalar-capillary-support-normalized-GCF} that $u^{T}(\cdot,t)$ satisfies
   \begin{eqnarray}\label{uT equ} \left\{
\begin{array}{llll}
  \p_t u^{T} &=&u^{T}-\frac{1}{\det \left(\n^2 (\ell u^{T})+\ell u^{T} \s \right)}   & \text{ in } ~\C_\theta ,\\ 
	\n_\mu u^{T}&=& 0 & \text{ on }~ \p \C_\theta  .\end{array}  \right.
\end{eqnarray}
Proposition \ref{prop-mono} implies
\begin{eqnarray}\label{entropy upper}
    \mathcal{E}_{\theta}(\wh\S_t )\leq \E_{\theta}(\wh\S_0 ), \quad\quad {\rm{for~all}}~t\in [0,+\infty),
    \end{eqnarray}
together with Proposition \ref{radius}, it yields
\begin{eqnarray}\label{aT}
    |a^{T}|\leq C,
\end{eqnarray} where $C>0$ only depends on $n$ and $\S_0$. 
Let $T$ be sufficiently large such that $u^{T}(\xi, 0) > 0$ for all $\xi\in \C_\theta$.  From \eqref{eq-rewrite-u-T-at-time-T} and Proposition \ref{lem entropy}, we have  \begin{eqnarray}\label{uT}
    u^{T}(\xi, T)>0, \quad\quad {\rm{for~all}}~\xi\in \C_{\theta}.
\end{eqnarray} 
Next, we \textbf{claim:} $$u^{T}(\xi, t)>0, ~~\text{for~ all}~(\xi,t)\in  \C_{\theta}\times[0, T].$$ 
We argue by contradiction. Suppose the claim is false, then there exists a point $\xi_{0}\in \C_{\theta}$ and $t_0\in (0, T)$ such that $u^{T}(\xi_{0}, t_{0})= 0$. 
By the first equation in \eqref{uT equ},  we have $\partial_{t}u^{T}(\xi_{0}, t_{0})<0$, then there exists a small constant $\delta>0$ and for all $t\in (t_{0}, t_{0}+\delta]$, we have $u^{T}(\xi_{0}, t)<0$. Repeating the same process at $t=t_0+\delta$, we iterative deduce that $u^{T}(\xi_{0}, T)<0$, which contradicts  with \eqref{uT}. Hence, the claim is true. 

Following the similar argument as in the proof of Proposition \ref{prop-mono}, for all $0\leq t\leq T$, we deduce   
\begin{eqnarray*}
    \frac{d}{dt}\left( \int_{\C_{\theta}}\log u^{T}(\xi, t)\ell(\xi) d\s\right)=- \int_{\C_{\theta}}\left(\sqrt{\frac{K(\xi, t)}{u^{T}(\xi,t)}}-\sqrt{\frac{u^{T}(\xi,t)}{K(\xi,t)}}\right)^{2}\ell(\xi) d\s,
\end{eqnarray*} by integrating over $t\in [0, T]$ and taking \eqref{eq-rewrite-u-T-at-time-T} into account yield
\begin{eqnarray}
  &&  \int_{\C_{\theta}}(\log u^{T}(\xi,0))\ell (\xi)d\s-\o_\theta \mathcal{E}_{\theta}(\widehat{\S}_T)=\int_{0}^{T}\int_{\C_{\theta}}\left(\sqrt{\frac{K(\xi, t)}{u^{T}(\xi,t)}}-\sqrt{\frac{u^{T}(\xi,t)}{K(\xi,t)}}\right)^{2}\ell(\xi) d\s dt \notag\\
    &\geq & \int_{0}^{T_{1}}\int_{\C_{\theta}}\left(\sqrt{\frac{K(\xi, t)}{u^{T}(\xi,t)}}-\sqrt{\frac{u^{T}(\xi,t)}{K(\xi,t)}}\right)^{2}\ell(\xi) d\s dt.    \label{ineq-estimate-on-entropy-refined}\end{eqnarray}
 Due to \eqref{aT}, when $T\to +\infty$, there holds 
 $$u^{T}(\xi, t)\to u(\xi, t) ~\text{ uniformly  for all } ~(t, \xi)\in [0,T_{1}]\times \C_{\theta}, $$ 
together with by letting $T\rightarrow +\infty$ in \eqref{ineq-estimate-on-entropy-refined}, we obtain
\begin{eqnarray*}
    \int_{\C_{\theta}}(\log u(\xi, 0))\ell(\xi) d\s- \o_\theta \mathcal{E}_{\infty}\geq  \int_{0}^{T_{1}}\int_{\C_{\theta}}\left(\sqrt{\frac{K(\xi, t)}{u(\xi,t)}}-\sqrt{\frac{u(\xi,t)}{K(\xi,t)}}\right)^{2}\ell(\xi) d\s dt.
\end{eqnarray*}
Thus \eqref{limit prop} holds for $t=0$. Changing $t=0$ to another time $t>0$ and letting $T_{1}\rightarrow +\infty$, we can show that \eqref{limit prop} holds for all $t>0$.
\end{proof}

With the preceding preparations in place, we now establish the crucial uniform positive lower bound and $C^0$ estimate for the solution to the flow \eqref{eq-scalar-capillary-support-normalized-GCF}.  A crucial ingredient is to show that for the solution $\S_{t}$ to flow \eqref{GCF-capillary-normalized}, the capillary entropy point $z_{e}(\widehat{\S
}_{t})$ of $\widehat{\S}_{t}$ satisfies
\begin{eqnarray}\label{distance}
    {\rm{dist}}\left(z_{e}(\widehat{\S}_{t}), \S_{t})\right)\geq \delta.
\end{eqnarray}

In order to establish \eqref{distance}, we introduce the notion of a general capillary convex body in $\ol{\RR^{n+1}_+}$, referred to as a \textit{$\theta$-capillary convex body} in Definition \ref{defn-angle}, by removing the smoothness assumption compared to Section \ref{sec2.1}. We denote the class of such bodies by $\G_\theta$, and define the corresponding entropy functional $\mathcal{F}_\theta(K)$ as \eqref{eq-entropy-on-K} for $K\in \G_\theta$, which extends both the class of capillary convex bodies $\K_\theta$ from Section \ref{sec2.1} and the functional $\E_\theta(\wh\S)$ from Section \ref{sec-4.2}, again without requiring smoothness. See Section \ref{sec-6}, and in particular Theorem \ref{thm-lim-Hausdorff-distance}, for further details.

\begin{thm}\label{thm-C0}
 Suppose that $u(\xi,t)$ is a positive solution to flow \eqref{eq-scalar-capillary-support-normalized-GCF} with $\theta \in (0,\pi/2)$  and  ${\rm{Vol}}(\wh\S_0 )={\rm{Vol}}(\widehat{\C_{\theta}})$. Then there exists  positive constants $C=C(n, \Sigma_{0})$ and $T_{0}=T_{0}(\S_{0})$, such that 
    \begin{eqnarray}\label{est-u-two-side-bound}
        \frac{1}{C}\leq u(\xi, t)\leq C, ~~~\text{ for all }\xi \in \C_\theta,~ t\geq T_{0}. \label{C0}
    \end{eqnarray}
\end{thm}
\begin{proof}
The uniform upper bound of $u(\xi, t)$ follows from Proposition \ref{radius}, Proposition \ref{prop-mono}, and \eqref{double-control}. Hence, it remains to establish a positive uniform lower bound for the solution.
 By Proposition \ref{pro-limit}, we obtain
 \begin{eqnarray*}
     \mathcal{E}_{\infty}\leq \frac{1}{\omega_{\theta}}\int_{\C_{\theta}}\left(\log u(\xi,t)\right)\ell(\xi) d\s\leq \mathcal{E}_{\theta}(\wh\S_t ), \quad \forall~ t\geq 0,
 \end{eqnarray*}
 that is,
 \begin{eqnarray}\label{t-limit}
     0\leq \mathcal{E}_{\theta}(\wh\S_t)-\frac{1}{\omega_{\theta}}\int_{\C_{\theta}}\left(\log u(\xi, t)\right)\ell(\xi) d\s \rightarrow 0, \quad {\rm{as}}~t\rightarrow +\infty, \end{eqnarray}
together with \eqref{ineq-stability-estimate} and \eqref{entropy upper}, it follows that the capillary entropy point $z_e(\wh\S_t)$ of $\wh\S_t$ satisfies
\begin{eqnarray*}
    |z_e(\wh\S_t)-0|\to 0, ~~~\text{ as } t\to +\infty.
\end{eqnarray*} In view of Theorem \ref{thm-lim-Hausdorff-distance} (2), it follows that  there exists a sufficiently large constant $T_{0}=T_{0}(\S_{0})>0$ such that  
\begin{eqnarray*}
    \text{dist}(0, \S_t) \geq \varepsilon_0, ~~~~\text{for all } t\geq T_0,
\end{eqnarray*}where $\varepsilon_0$ is a positive constant, depending only on $n$ and $\S_0$. Thus, in turn, implies that there exists a constant $C=C(n,\S_0)>0$ such that $$u(\xi, t) \geq \frac{1}{C}, ~~~\text{ for all }\xi \in \C_\theta,~ t \geq T_0.$$ This completes the proof.
\end{proof}

\subsection{Curvature estimates}
In this subsection, first, we will derive evolution equations for the capillary support function $u$ and Gauss curvature $K$ along the normalized flow \eqref{GCF-capillary-normalized}. Subsequently, we establish the uniform estimates for the solution of the normalized flow \eqref{GCF-capillary-normalized} when $\theta\in (0,\pi/2)$.  As in Section \ref{sec3}, we remark that the two-sided estimates for the Gauss curvature of the solution to flow \eqref{GCF-capillary-normalized} hold for the whole contact angle $\theta \in (0, \pi)$. Finally, by combining with all the known results, we complete the proof of Theorem \ref{GCF-normal to soliton}.

\begin{prop}
    Along flow \eqref{Support-Normalized},  we have the evolution equations.
    \begin{enumerate}
        \item  The capillary support function $u$ satisfies
        \begin{eqnarray}\label{eq-Lu}
            \mathcal{L} u=2  K a^{ij}\n_{i}u \n_{j}\ell-(n+1) K+  KH u+u,
        \end{eqnarray}
   and   
        \begin{eqnarray}\label{eq-Neumann-u}
            \n_{\mu}u=0, ~~\text{ on } \partial\C_{\theta}.
        \end{eqnarray}
        \item The Gauss curvature $K$ satisfies
        \begin{eqnarray}\label{eq-LK}
            \mathcal{L}K=2 Ka^{ij}\n_{i}K\n_{j}\ell+ K^{2}H-n K,
        \end{eqnarray}
and 
    \begin{eqnarray}\label{eq-Neumann-K}
        \n_{\mu}K=0,  ~~\text{ on } \partial\C_{\theta}.
    \end{eqnarray}        
    \end{enumerate}
\end{prop}
\begin{proof}
The proofs are similar to those of Proposition~\ref{evo-h} and Proposition~\ref{evo of K}, respectively, and are therefore omitted.
\end{proof}

First, we obtain an upper bound of $K$.
\begin{prop}\label{prop-F-lower}
Let $\S_t$ be the solution of flow \eqref{GCF-capillary-normalized} with $\theta \in (0,\pi/2)$.  If ${\rm{Vol}}(\wh\S_0 )={\rm{Vol}}(\widehat{\C_{\theta}})$, then there exists a positive constant $C$, depending only on $n$ and $\S_{0}$, such that
    \begin{eqnarray}\label{K-upper}
        K(\xi,t)\leq C, ~~\text{ for all } (\xi,t)\in \C_\theta\times[0, +\infty).
    \end{eqnarray}
\end{prop}
\begin{proof}
Let $T\in [0, T^{\ast})$ and $\widetilde{X}(\xi, \tau)$ be the solution to flow \eqref{GCF-capillary-2} in $[0, T]$. For all $\tau \in [0, T]$, let $\rho_{-}(\widehat{\widetilde \S}_{\tau}, \theta)$ and $\rho_{+}(\widehat{\widetilde\S}_{\tau}, \theta)$ be the capillary inner and outer radius of $\widehat{\widetilde{\S}}_{\tau}$ respectively.  For brevity, we abbreviate $\rho_{-}(\widehat{\widetilde\S}_{\tau}, \theta)$ as $\rho_{1}(\tau)$ and $\rho_+(\widehat{\widetilde\S}_{\tau}, \theta)$ as $\rho_{2}(\tau)$.  From Theorem \ref{main-existence}, there exists a point $\widetilde z_{0}\in \ov{\RR^{n+1}_{+}}$ such that
\begin{eqnarray*}
   \widehat{\C_{\rho_{1}(T), \theta}(\widetilde z_{0})}\subset \widehat{\widetilde\S}_{\tau}, \quad {\rm{for~all}}~\tau \in [0, T].
\end{eqnarray*}
Picking $c_{0}=\frac{1}{2}\rho_{1}(T)$ in \eqref{K-upper-bound}, we obtain 
\begin{eqnarray*}
    \widetilde{K}(\xi, T)\leq \max\left\{ 2^{n+1}\left(\frac{n+1}{n}\right)^{n}\frac{ \widetilde u_{\widetilde z_{0}}(\xi,T)
    }{ \rho_{1}(T)^{1+n}}, \frac{ \widetilde{u}_{\widetilde z_{0}}(\xi, T)\widetilde{K}(\xi, 0)}{\rho_{1}(T)}
   \right \}.
\end{eqnarray*}
where $\widetilde{u}_{\widetilde z_{0}}(\xi, T)$ is the capillary support function of $\widehat{\widetilde{\S}}_{T}$ with respect to the point $\widetilde z_{0}$.  Denote $t\coloneqq \frac{1}{n+1}\log \left(\frac{{\rm{Vol}}(\widehat{\C_{\theta}})}{{\rm{Vol}}(\widehat{\widetilde\S}_{T})}\right)$.   
In view of the rescaling, we obtain 
\begin{eqnarray}\label{upper-K}
    K(\xi, t)&\leq& \max\left\{ 2^{n+1}\left(\frac{n+1}{n}\right)^{n} \frac{2\rho_{+}(\widehat{\S}_{t}, \theta)}{\rho_{-}(\widehat{\S}_{t}, \theta)^{1+n}}, \frac{2\rho_{+}(\widehat{\S}_{t}, \theta)}{\rho_{-}(\widehat{\S}_{t}, \theta)} K(\cdot, 0)
    \right\}.
\end{eqnarray}
Combining Proposition \ref{radius}, Proposition \ref{prop-mono} and Proposition \ref{volume-preserving}, we have 
\begin{eqnarray}\label{upper-ratio}
    \frac{\rho_{+}(\widehat{\S}_{t}, \theta)}{\rho_{-}(\widehat{\S}_{t}, \theta)}\leq \frac{C^{2}e^{(n+1)\E_{\theta}(\widehat{\S}_{t})}}{ {\rm{Vol}}(\widehat{\S}_{t})}\leq \frac{C^{2}e^{(n+1)\E_{\theta}(\widehat{\S}_{0})}}{{\rm{Vol}(\widehat{\C_{\theta}})}},
\end{eqnarray}
and 
\begin{eqnarray*}
    \frac{\rho_{+}(\widehat{\S}_{t}, \theta)}{\rho_{-}(\widehat{\S}_{t}, \theta)^{n+1}}\leq \frac{C^{n+2} e^{\left(1+n(n+1)\right)\E_{\theta}(\widehat\S_{0})}}{{\rm{Vol}}(\widehat{\C_{\theta}})^{n+1}},
\end{eqnarray*}
together with \eqref{upper-K} and  \eqref{upper-ratio}, we complete the proof of \eqref{K-upper}.


\end{proof}

Next, we establish a positive lower bound for $K$ by applying the maximum principle. This approach relies crucially on the lower bound of the capillary support function $u$ provided in Theorem \ref{thm-C0}, along with its homogeneous Neumann boundary condition.

\begin{prop}\label{pro-F-upper}
Let $\S_t$ be the solution of flow \eqref{GCF-capillary-normalized}  with $\theta \in (0,\pi/2)$.  If ${\rm{Vol}}(\wh\S_0 )={\rm{Vol}}(\widehat{\C_{\theta}})$, then there exists a positive constant $C$ depending on $n$ and $\S_{0}$, such that
    \begin{eqnarray}\label{K-lower}
        K(\xi,t)\geq C,~~\text{ for all } (\xi,t)\in \C_\theta\times[T_0, +\infty). 
    \end{eqnarray}
\end{prop}
\begin{proof}
 Define
   \begin{eqnarray*}
       P(\xi,t) \coloneqq \log (Ku^{\gamma}), ~~~\text{ for }\xi\in \C_\theta,~~ t\geq  T_0,
   \end{eqnarray*}
   where $\g>0$ is a constant to be determined later.
From \eqref{eq-Neumann-u} and \eqref{eq-Neumann-K}, we have
   \begin{eqnarray}\label{P-neumann}
        \n_{\mu}P=0, \quad  {\rm{on}}~ \partial\C_{\theta}. 
   \end{eqnarray} 
   Then we know $P$ attains its minimum value at some interior point, say $\xi_0\in \C_\theta\setminus\p \C_\theta$. Next, all the calculations are done at the minimum value point $\xi_0$. First,  we have
   \begin{eqnarray}\label{one-deri}
   0=\n_i P=    \frac{\n_i K}{K}+\frac{\gamma \n_{i}u}{u},
   \end{eqnarray}
   and by using Eq. \eqref{eq-Lu}, \eqref{eq-LK}, and  \eqref{one-deri}, it follows that 
   \begin{eqnarray}\label{two-deri}
      0&\geq & \mathcal{L}P\notag  = \frac{1}{K}\mathcal{L}K+\frac{\gamma}{u}\mathcal{L}u+ \ell Ka^{ij} \left(\frac{\n_{i}K\n_{j}K}{K^{2}}+\gamma\frac{\n_{i}u \n_{j}u}{u^{2}} \right)\notag  \\
       &\geq &\frac{1}{K}\mathcal{L}K+\frac{\gamma}{u}\mathcal{L}u\geq (\gamma-n)-\frac{(n+1)\gamma }{u}K.     
   \end{eqnarray}
By choosing $\gamma = 2 + n$ and substituting \eqref{C0} into \eqref{two-deri}, we conclude that
\begin{eqnarray*}
P(\xi, t) \geq P(\xi_0, t) \geq \log\left(\frac{u^{n+3}}{(n+1)(n+2)}\right).
\end{eqnarray*}
Applying \eqref{est-u-two-side-bound} once more, we obtain \eqref{K-lower}.
   \end{proof}

Once two-sided estimates for the Gauss curvature $K$ are established, we obtain corresponding bounds on the principal radii of $\S_t$, a similar result to \cite{And00, AGN} and \cite{GN} for the closed setting. We remark that the contact angle range $\theta < \pi/2$ is crucial to ensure the applicability of the boundary maximum principle.  
\begin{prop}\label{cur-est}
Let $h(\xi, t)$ be the solution of flow \eqref{Support-Normalized} with $\theta \in (0,\pi/2)$. 
 Suppose that ${\rm{Vol}}(\wh\S_0 )={\rm{Vol}}(\widehat{\C_{\theta}})$, then there exist A  positive constant $C$,  depending only on $n$ and $\S_{0}$, such that
    \begin{eqnarray}\label{principal-radius-bi-estimates}
        \frac{1}{C} I\leq (h_{ij}+h\s_{ij})\leq C I, ~~\text{ for all } (\xi,t)\in \C_\theta\times[T_0, +\infty),
    \end{eqnarray}
  where $T_{0}$ is a positive constant given in Theorem \ref{thm-C0}.
\end{prop}
\begin{proof}
  Since $\n^{2}h+h\s$ is positive definite, and the upper and lower bound of $K$ implies that it suffices to prove the upper bound of the eigenvalues of $\n^{2}h+h\s$.   Similar to the proof of Proposition \ref{curvature est}, we consider the function 
    \begin{eqnarray*}
        \Phi\coloneqq \Delta h+n h.
    \end{eqnarray*}
Recall from \eqref{boundary of Phi}, there holds  $$\n_{\mu}\Phi
    \leq 0 \text{ on } \partial\C_{\theta}.$$ Thus, we know that $\Phi$ attains its maximum value within $\C_\theta$ at some interior point, say $\xi_{0}\in \C_{\theta}\setminus \partial \C_{\theta}$. By performing calculations similar to those in Proposition \ref{curvature est}, at $\xi_0$, we have
    \begin{eqnarray}
      0\leq  \mathcal{L}\Phi&=&\Phi-\ell K\Phi H- \left[(n\ell-1) nK 
     - \ell K a^{kp} \n_{i}a_{pq} a^{ql}\n_{i}a_{kl} \right.\notag  \\
     &&\left.- \ell K a^{pq}\n_{i}a_{pq}a^{kl}\n_{i}a_{kl} -2  \n\ell \cdot \n K\right].~~~~~~~~~~\label{Phi evo}
    \end{eqnarray}
Together with \eqref{Kalpha upper} and  \eqref{K-upper}, we get
  \begin{eqnarray}
 &&  K \left[ -  \ell \sum_{k,l,i} a^{kk}a^{ll}(\n_{i}a_{kl})^2- \a \ell  \sum_i \left(\sum_k a^{kk}\n_{i}a_{kk} \right)^2-2\cos\theta \<e_i,e\> a^{kk} \n_{i}a_{kk} \right] \notag  \\&& \leq  C K \leq C.~~~\label{Kalpha-2}
\end{eqnarray}
By the Newton-Maclaurin inequality, we obtain
\begin{eqnarray}\label{H-low bound}
   H= \frac{\sigma_{n-1}(A)}{\sigma_{n}(A)}\geq n\left(\frac{\sigma_{1}(A)}{n\sigma_{n}(A)}\right)^{\frac{1}{n-1}}=n^{\frac{n-2}{n-1}}K^{\frac{1}{n-1}}\Phi^{\frac{1}{n-1}}. 
\end{eqnarray}
Inserting \eqref{Kalpha-2} and \eqref{H-low bound} into  \eqref{Phi evo}, we have 
\begin{eqnarray*}
 0\leq \Phi -n^{\frac{n-2}{n-1}}\ell (K\Phi)^{\frac{n}{n-1}}-(n\ell-1)nK-C,   
\end{eqnarray*}
Using \eqref{K-upper}, \eqref{K-lower} again, we conclude that $$\Phi\leq C,$$ which yields the upper bound in \eqref{principal-radius-bi-estimates} and the lower bound follows together with \eqref{K-lower}. This completes the proof.
\end{proof}

We are now ready to complete the proof of the second part of Theorem~\ref{main-thm}, namely Theorem~\ref{GCF-normal to soliton}.
\subsection{Proof of Theorem \ref{GCF-normal to soliton}}
\begin{proof}
In view of Theorem \ref{thm-C0}, Proposition \ref{prop-F-lower}, Proposition \ref{pro-F-upper} and Proposition \ref{cur-est}, we conclude that the solution $u$ of flow \eqref{eq-scalar-capillary-support-normalized-GCF} satisfies 
\begin{eqnarray*}
    \|u(\cdot, t)\|_{C^{2}(\C_{\theta})}\leq C,
\end{eqnarray*}
 where the positive constant $C$ only  depends on $n$ and $\S_{0}$. Thus by the standard parabolic fully nonlinear equation theory with Neumann boundary value condition (see, e.g., \cite[Theorem 6.1, Theorem 6.4 and Theorem 6.5]{Dong}, and also \cite[Theorem 14.23]{Lieberman-book-parabolic}),  for any integer $k\geq 2$ and $\alpha\in (0,1) $, there holds
\begin{eqnarray}\label{est}
    \|u(\cdot, t)\|_{C^{k, \alpha}(\C_{\theta})}\leq C.
\end{eqnarray}
Given any $T>0$, and a number sequence $\{t_{j}\}_{j\geq 1}$ with $t_j\nearrow + \infty$, consider the sequence functions 
\begin{eqnarray}\label{uj}
u_{j}(\xi, t)\coloneqq u(\xi, t+t_{j})=\frac{h(\xi, t+t_j)}{\ell(\xi)}.  
\end{eqnarray}
For each $j\geq 1$, we see that $\{u_{j}\}_{j\geq 1}$ are still the solutions of flow  \eqref{eq-scalar-capillary-support-normalized-GCF} and satisfy  \eqref{est}. This implies $u_{j}(\xi,t)$ (up to a subsequence) converges in $C^{\infty}$ topology to a limit function $u_{\infty}(\xi,t)$ for all $(\xi,t)\in \C_\theta \times [-T,T]$, and $u_{\infty}$  satisfies the flow  \eqref{eq-scalar-capillary-support-normalized-GCF}.  Moreover, Theorem \ref{thm-C0} implies that $u_{\infty}>0$. 
On the other hand, together with \eqref{t-limit}, it implies that $u_{\infty}$  satisfies 
    \begin{eqnarray*}
     \int_{\C_{\theta}}\left(\log u_{\infty}(\xi,t)\right)\ell(\xi)d\s=\lim\limits_{j\rightarrow +\infty} \int_{\C_{{\theta}}}\left(\log u_{j}(\xi, t)\right)\ell (\xi)d\s=\o_\theta \mathcal{E}_{\infty},
    \end{eqnarray*}
    together with the proof of Proposition \ref{prop-mono}, it yields
    \begin{eqnarray*}
   0=\frac{d}{dt}\left( \int_{\C_{\theta}}(\log u_{\infty}(\xi, t))\ell(\xi) d\s\right)=- \int_{\C_{\theta}}\left(\sqrt{\frac{u_{\infty}(\xi,t)}{K(\xi, t)}}-\sqrt{\frac{K(\xi, t)}{u_{\infty}(\xi, t)}}\right)^{2}\ell(\xi) d\s,
   \end{eqnarray*}
which in turn implies that 
    \begin{eqnarray*}
        \frac{u_{\infty}(\xi, t)}{K(\xi,t)}=\frac{K(\xi, t)}{u_{\infty}(\xi, t)},
    \end{eqnarray*}
combining with the fact that $u_{\infty}(\xi, t)$ satisfies the flow \eqref{eq-scalar-capillary-support-normalized-GCF}, we obtain $$\p_t u_{\infty} (\xi, t)=0.$$ Hence $u_{\infty}(\cdot)\coloneqq u_{\infty}(\cdot, t)$ satisfies 
    \begin{eqnarray*}
    \left\{
\begin{array}{llll}
 u_{\infty} &=&\frac 1 {\det \left(\n^2 (\ell u_\infty)+\ell u _\infty\s \right)}   & \text{ in } ~\C_\theta,\\ 
	\n_\mu u_{\infty} &=& 0 & \text{ on }~ \p \C_\theta, \end{array}  \right.
\end{eqnarray*} which is equivalent to \eqref{eq-soliton-normalized-GCF}, by  taking into account \cite[Proposition 2.6]{MWWX}.  
Therefore, we obtain that after passing to a subsequence, the solution $\{u_j\}_{j\geq 1}$ of \eqref{eq-scalar-capillary-support-normalized-GCF} smoothly converges to a smooth function $u_\infty$, which satisfies Eq. \eqref{eq-soliton-normalized-GCF}. 

Finally, we prove that $u(\xi,t)$ globally converges in the $C^{\infty}$-topology to $u_{\infty}(\xi, t)$ as $t\rightarrow +\infty$. We proceed by contradiction. If not, then there there exist $l\in \mathbb{N}$ and a sequence $\{t_{j}\}_{j\geq 1}$ with $t_j\nearrow + \infty$, such that 
\begin{eqnarray}\label{m-deriv}
    \sup\limits_{\C_{\theta}}|(\n^{l}u)(\xi, t_{j})-(\n^{l}u_{\infty})(\xi)|\geq \tau,
\end{eqnarray}
 for some positive constant $\tau$. On the other hand, we consider the sequence $u_{j}(\xi,t)\coloneqq u(\xi,t+t_j)$ defined in way of \eqref{uj}, we see that $u_{j}(\xi, t)$ (up to a subsequence) converges in $C^{\infty}$ topology to $u_{\infty}(\xi)$ on $\C_{\theta}\times \{0\}$. This implies  $u_{t_{j}}(\xi, 0)$ converges in $C^{\infty}$ topology to $u_{\infty}(\xi)$ on $\C_{\theta}$, which is  a contradiction to \eqref{m-deriv}. Hence, we complete the proof.

\end{proof}

\section{Perspective}\label{sec5}
 In this section, we discuss some generalizations of flow \eqref{c-Gauss curvature flow} to the capillary $\a$-power Gauss curvature flow in $\ol{\RR^{n+1}_+}$. We briefly mention the key steps involved, omitting full details for brevity.
Chow \cite{Chow85} extended Firey's result and also Tso \cite{Tso85} to the $\alpha$-power Gauss curvature flow for convex closed hypersurface in $\RR^{n+1}$:
\begin{eqnarray}\label{alpha-GCF-2}\p_t X= -K^\a\nu,
\end{eqnarray}for $\alpha>0$ and showed that the limiting point is a round point if $\a=\frac 1 n$.  
Thus, the remaining problem for the $\a$-power Gauss curvature flow \eqref{alpha-GCF-2} is to characterize the limiting shape of \eqref{alpha-Gauss curvature flow} for general $\a\neq \frac 1 n$. When $\alpha=\frac{1}{n+2}$, Andrews \cite{And94-1} proved that such a flow will converge to an ellipsoid shape point, see also Sapiro and Tannenbaum \cite{ST}  for the case $n=1$. 
When $n=2$, Andrews \cite{And99} solved Firey's conjecture completely, i.e., when $\alpha=1$, the flow converges to a round point. Andrews-Chen \cite{AC} showed that when $\alpha\in [\frac{1}{2}, 1]$, this conclusion also holds. Their methods highly depend on the geometry of surfaces in $\mathbb{R}^{3}$ and seem not to work in high dimensions.  Subsequently, when $n\geq 2$,  Andrews \cite{And00} proved that if $\alpha\in (\frac{1}{n+2}, \frac{1}{n})$, the normalized flow will converge to a self-similar solution of flow \eqref{alpha-GCF-2}. A closed hypersurface $M$ is said to be a self-similar solution of flow \eqref{alpha-GCF-2} if it satisfies 
\begin{eqnarray}\label{alpha-GCF similar solution}
    K^{\alpha}=\<X, \nu\>.
\end{eqnarray}
If $\a>\frac 1 {n+2}$, the convergence of flow \eqref{alpha-GCF-2} starting from a strictly convex closed hypersurface to the soliton  Eq. \eqref{alpha-GCF similar solution} and the classification of the strictly convex solution to  Eq. \eqref{alpha-GCF similar solution} were settled by Andrews-Guan-Ni \cite{AGN} and Brendle-Choi-Daskalopoulos \cite{BCD}, respectively.

Our method allows a similar result for the $\a$-power Gauss curvature flow of evolving strictly convex capillary hypersurfaces $\S_t\coloneqq X(M,t)\subset\ol{\RR^{n+1}_+}$ if $\a>0$. To be more precise,  we consider that the capillary $\alpha$-power Gauss curvature flow in $\ol{\RR^{n+1}_+}$, which is defined as:
\begin{eqnarray}\label{cap-alpha-Gauss curvature flow}
   \p_t X(x,t)=-K^\a (x,t)\tilde\nu(x,t), ~~~(x,t)\in M\times[0, T'),
\end{eqnarray} and its corresponding normalized capillary $\a$-power Gauss curvature flow: 
\begin{eqnarray}\label{c-alpha-GCF-capillary-normalized}
\left\{
\begin{array}{llll}
\p_t X(\cdot, t) &=&X(\cdot,t)- \frac{(n+1) {\rm{Vol}}(\wh{\C_\theta}) K^\alpha(\cdot, t)}{ \int_{\C_\theta} K^{\a-1}(\xi,t) (\sin^2\theta+\cos\theta\<\xi,e\>) d\xi } \tilde\nu(\cdot,t) ,  &
\hbox{ in } M\times [0,+\infty),\\
\<\tilde\nu ,e\> &=&0
 & \hbox{ on }\partial M\times [0,+\infty),\\
X(\cdot,0)&=& X_0(\cdot) & \hbox{ on } M, 
\end{array} \right.
\end{eqnarray}where $X_0\colon   M\to \ol{\RR^{n+1}_+}$ is a strictly convex capillary hypersurface. As discussed in the introduction for the case $\alpha = 1$, flow \eqref{cap-alpha-Gauss curvature flow} is also closely related to the capillary $L_p$-Minkowski problem with $p = 1-\frac{1}{\alpha}$.  
 
\begin{rem}     
Along flow \eqref{c-alpha-GCF-capillary-normalized}, then the enclosed volume ${\rm{Vol}}(\widehat{\S}_{t})$ is preserved invariant for all $t\geq 0$, if ${\rm{Vol}}(\wh\S_0)={\rm{Vol}}(\wh{\C_\theta})$.
\end{rem}
Motivated by prior works on closed convex hypersurfaces, including Andrews–Guan–Ni \cite{AGN} for $\alpha > \frac{1}{n+2}$, Guan–Ni \cite{GN}, Firey \cite{Firey} for the case $\alpha = 1$, and more recently B\"or\"oczky–Guan \cite[Eq.~(1.6)]{BG} in the anisotropic setting, we introduce the capillary $\alpha$-entropy functional for general $\alpha > 0$.

\begin{defn}[Capillary $\a$-entropy functional]
Let $\wh\S\in \K_\theta$ be a capillary convex body,  the capillary $\a$-entropy functional $\E_{\a,\theta}$ for $\wh\S$ is defined by
\begin{eqnarray}\label{cap-entropy-alpha}
    \mathcal{E}_{\a,\theta}(\wh\S)\coloneqq \sup_{z_0\in {\rm{int}}(\widehat{\partial\S})}    \E_{\a,\theta}(\wh\S,z_0),
\end{eqnarray} and 
\begin{eqnarray*} 
    \E_{\a,\theta}(\wh\S,z_0)\coloneqq \begin{cases}
        \frac{\a}{\a-1} \log\left( \frac 1 {\o_\theta} \int_{\C_\theta} u_{z_0}(\xi)^{1-\frac 1 \a} \ell(\xi) d\sigma \right),~~ &\a\neq 1,\\
        \frac 1 {\omega_\theta}\int_{\C_\theta} (\log u_{z_0}(\xi)) \ell(\xi)  d\s, ~~ &\a=1,
    \end{cases}
\end{eqnarray*}where $d\s$ is the standard area element on $\C_\theta$ and $\omega_\theta$ is the capillary area of $\C_\theta$.  
\end{defn}

Building on the analysis developed in this paper and drawing inspiration from the ideas of Andrews–Guan–Ni \cite{AGN}, with the aid of concept \eqref{cap-entropy-alpha}, we expect to establish the following result in a forthcoming work.

 \begin{thm} Let $\S_0$ be a strictly convex capillary hypersurface in $\ol{\R^{n+1}_+}$ and $\theta\in(0,\pi/2)$. If $\alpha>0$, then the flow \eqref{cap-alpha-Gauss curvature flow} has a smooth solution $\S_t\coloneqq X(M,t)$, which is strictly convex capillary hypersurface for all $t\in[0, T^*)$, where $T^*\coloneqq \frac 1 {\a n+1}\left(\frac{{\rm{Vol}}(\widehat\S_0)}{{\rm{Vol}}(\wh\C_{\theta})} \right)^{\frac{\a n+1}{n+1}}$. As $t\to T^*$, then the capillary hypersurfaces $\S_t$ shrink to a point $p\in \p{\RR^{n+1}_{+}}$.
 Moreover,  if $\a>\frac 1{n+2}$, the solution of normalized flow \eqref{c-alpha-GCF-capillary-normalized} converges to a smooth, strictly convex capillary hypersurface $\S_\infty\subset \ol{\RR^{n+1}_+}$ as  $t\to +\infty$. 
 Furthermore, $\S_\infty$ satisfies: \begin{eqnarray}\label{soliton-alpha-GCF}\left\{
\begin{array}{llll}
	K^\a  &=& \lambda \frac{\<X,\nu\> } {1+\cos\theta\<\nu, e\>} & \text{  in } ~~ \S_\infty,\\ 
	\<\nu,e\>&=& -\cos\theta & \text{ on }~~ \p \S_\infty,\end{array} \right.
\end{eqnarray} where $\lambda\coloneqq \frac {1}{(n+1) {\rm{Vol}}(\wh\C_\theta)} \int_{\C_\theta} K^{\a-1}(\xi) \left(\sin^2\theta+\cos\theta\<\xi,e\>\right) d\xi$.
\end{thm}

The classification of solutions to Eq. \eqref{soliton-alpha-GCF} remains an open problem. It is natural to conjecture that, for $\a>\frac 1{n+2}$ and $\theta \in (0, \pi)$, the only strictly convex capillary hypersurface to Eq. \eqref{soliton-alpha-GCF} (without loss of generality, assuming the Lagrange multiplier $\lambda\equiv 1$) are the unit spherical cap, namely $\S_\infty \equiv \C_\theta$. We conclude the paper by formulating the following equivalent conjecture.

\begin{conjecture}
{\it   Given $\alpha > \frac{1}{n+2}$ and $\theta \in (0, \pi)$, suppose that $h$ is a positive, strictly convex function in $\C_\theta$ and solves
\begin{eqnarray}
\left\{
\begin{array}{llll}
	h\det^\a(\n^2 h+h\s) = \ell & \text{  in } ~~ \C_\theta , \\ 
	\n_\mu h= \cot\theta h & \text{ on }~~ \p \C_\theta,\end{array} \right.
\end{eqnarray}
then $h=\ell$.}  
 \end{conjecture}

\ 
\

\section{Appendix: \texorpdfstring{$\theta$}{}-capillary convex body} \label{sec-6}
In this section, we introduce the concept of $\theta$-capillary convex bodies (i.e., Definition \ref{defn-angle}) without assuming smoothness, in contrast to Section~\ref{sec2.1}, and define the associated capillary entropy functional as \eqref{eq-entropy-on-K}. We then investigate the limiting behavior of capillary convex bodies under convergence with respect to the Hausdorff distance. We believe that the results presented in this section are of independent interest.

In order to extend the definition of the contact angle for non-smooth convex body in $\ov{\RR^{n+1}_{+}}$ with one part of boundary lying on $\partial{\RR^{n+1}_{+}}$,  we first introduce the normal cone, which is a standard notion, see, e.g., \cite[Section 2]{LWW-2023}.
\begin{defn}[Normal cone]
Let $K$ be a convex body (i.e., a compact, convex set with non-empty interior) contained in $\ol{\RR^{n+1}_+}$. The normal cone of $K$ at the point $x\in \partial K$, denoted $\mathcal N_x K$, is defined by
\begin{eqnarray*}
\mathcal N_x K\coloneqq \left\{ \zeta \in \SS^n \mid \<\zeta,y - x\> \leq 0,~~\text{ for all } y \in K \right\}.
\end{eqnarray*}
\end{defn}
The classical support function $h_{z}$ of $K$ with respect to the point $z\in {\rm{int}}(K)\cup (K\cap \p\RR^{n+1}_+)$ is defined by 
\begin{eqnarray*}
    h_{z}(\zeta)\coloneqq \sup\limits_{y\in K}\<y-z, \zeta\>, \quad  {\rm{for~all}}~\zeta\in \SS^{n}.
    \end{eqnarray*}
To maintain notational consistency, as in Section \ref{sec2.1}, we shall continue to denote 
\begin{eqnarray*}
    u_{z}(\zeta)\coloneqq \frac{h_{z}(\zeta)}{\ell(\zeta)},
    \quad  {\rm{for~all}}~\zeta\in \SS^{n}.
\end{eqnarray*}where $\ell(\zeta)=1-\cos\theta \<\zeta, E_{n+1}\>$.
\begin{defn}\label{defn-general-capillary-convex-body}
  Let $K$ be a convex body contained in $\overline{\mathbb{R}^{n+1}_+}$. The boundary of $K$ consists of two components: one is $K_{u}\coloneqq \p  K\cap \RR^{n+1}_{+}$, and the other, denoted by $K_{f}\coloneqq \p K\cap \partial\ov{\RR^{n+1}_{+}}$, lies on $\partial \ov{\mathbb{R}^{n+1}_+}$. These two components share a common boundary, denoted by $\partial K_f$. We require that $K_u$ is nonempty and that $K_f$ contains interior points (relative to $\partial \mathbb{R}^{n+1}_+$).
We refer to such a $K$ as a convex body restricted to $\overline{\mathbb{R}^{n+1}_+}$. 
    \end{defn}
Denote by $\mathcal{G}$ the set of all convex bodies $K$ contained in $\overline{\mathbb{R}^{n+1}_{+}}$ that satisfy Definition \ref{defn-general-capillary-convex-body}. Next, we define the contact angle of a convex body $K \in \mathcal{G}$ along the contact boundary $\partial K_f$ within $\partial \mathbb{R}^{n+1}_{+}$.

\begin{defn}[$\theta$-capillary convex body]\label{defn-angle}
  Let $K\in \mathcal G$ and $x\in \p K_f$, we define the contact angle $\vartheta=\vartheta (x)\in (0,\pi)$ of $K$ at the point $x \in \partial K_{f}$  by 
  \begin{eqnarray*}
      \cos \vartheta= \sup\limits_{\zeta\in \mathcal N_{x}K}\<\zeta, E_{n+1}\>.
  \end{eqnarray*}
For all $x\in \p K_f$, if $\vartheta(x)\leq \theta$ for some fixed constant $\theta\in(0, \pi)$, then we call such $K$ the $\theta$-capillary convex body in $\ol{\RR^{n+1}_+}$. 
  \end{defn}
Denote by $\mathcal{G}_\theta$ the set of all $\theta$-capillary convex bodies in $\ol{\mathbb{R}^{n+1}_{+}}$.
Given $\theta \in (0,\pi\slash 2]$, we introduce an entropy functional for $K\in \mathcal{G}_\theta$, given by 
   \begin{eqnarray}\label{eq-entropy-on-K}
         \mathcal F_{\theta}(K)\coloneqq  \sup\limits_{z\in {\rm{int}}
         (K_{f})} \frac{1}{\omega_{\theta}}\int_{\SS^{n}_{\theta}}\left(\log u_{z}(\zeta)\right)\ell(\zeta)d\s,
     \end{eqnarray}
 In contrast to the entropy functional for classical convex bodies introduced by Guan-Ni \cite{GN} and Firey \cite{Firey}, which is defined on $\SS^{n}$, the entropy functional $\mathcal F_\theta(K)$ considered in this setting is only defined on the unit spherical cap
 $$\SS^{n}_{\theta}\coloneqq \C_{\theta}+\cos\theta E_{n+1}= \left\{\zeta\in \SS^{n}  ~\mid ~  \<\zeta, E_{n+1}\>\geq \cos\theta \right\}.$$ Such a modification is similar to Definition \ref{def-cap-entropy} for the capillary convex body  $\wh\S\in \K_\theta$. 
 This definition is well-suited for $\theta$-capillary convex bodies in $\G_\theta$, as demonstrated by the following property, which is crucial for proving Theorem \ref{thm-lim-Hausdorff-distance} later.

\begin{prop}\label{prop-small angle}
Let  $K \in \mathcal{G}_{\theta}$ with $\theta\in (0, \pi/2)$. Then $\mathcal F_{\theta}(K)$ is well-defined. Moreover, $\mathcal{F}_{\theta}(K)$ is attained by a unique interior point $z_{e}\coloneqq z_{e}(K)\in{\rm{int}}(K_{f})$.
 \end{prop}

\begin{proof}
 We start by showing that $\mathcal{F}_{\theta}(K)$ is well-defined. It suffices to show that
\begin{eqnarray}\label{ineq-two-sides-est-F-theta}
    -C\leq \mathcal{F}_{\theta}(K)\leq C,
\end{eqnarray} 
where $C$ is a positive constant depending only on $n, \rho_{+}(\widehat{K})$ and ${\rm{Vol}}(K)$.  In fact, the upper bound in \eqref{ineq-two-sides-est-F-theta} follows directly  from 
\begin{eqnarray*}
    \mathcal{F}_{\theta}(K)\leq \frac{1}{\omega_{\theta}}\int_{\SS^{n}_{\theta}} \log \left(2\rho_{+}(K) \right)\ell(\zeta)
d\sigma =\log \left(2\rho_{+}({K}) \right).
\end{eqnarray*}
Next, we proceed to show the lower bound in \eqref{ineq-two-sides-est-F-theta}. Let $\widetilde K$ be the bounded domain by reflecting $K$ with respect to the hyperplane $\p\RR^{n+1}_+$. We \textbf{claim} that $\widetilde K$ is a bounded closed convex body in $\RR^{n+1}$. Since  $\bigcup\limits_{x\in K_{u}}\mathcal{N}_{x}K$ is  a connected component of $\SS^{n}$ and $E_{n+1}\in \bigcup\limits_{x\in K_{u}}\mathcal N_{x}K$, and from $K\in\mathcal{G}_{\theta}$, we have $\vartheta(x) \leq \theta$ for all $x\in \partial K_{f}$ with $\theta\in (0, \pi/2)$, then we conclude that
$$\ov{\bigcup\limits_{x\in K_{u}}\mathcal N_{x} K}\subset \SS^{n}_{\theta},$$ 
which implies the claim. We denote by $\widetilde z_{s}$ the classical Santal\'o point of  $\widetilde {K}$, using the symmetry of $\widetilde K$ with respect to $\p\RR^{n+1}_+$ and \eqref{sym}, it implies   
\begin{eqnarray}\label{eq-santalo-wilde-K-interior}
\widetilde z_{s}\in {\rm{int}}(K_{f}).
\end{eqnarray}
From \cite[Page 547]{Sch} or \cite[(1.1)]{San}, there holds
\begin{eqnarray}\label{eq-volume-polar-body}
    {\rm{Vol}}(\widetilde K_{ \widetilde z_{s}}^{\ast})=\frac{1}{n+1}\int_{\SS^{n}}\frac{1}{h_{ \widetilde z_{s}}^{n+1}(\zeta)}d\sigma,
\end{eqnarray}
where $\widetilde K_{ \widetilde z_{s}} ^{\ast}$ and $h_{\widetilde z_{s}}$ denote the polar body and support function of $\widetilde K$ with respect to the point $\widetilde z_{s}$, respectively. In view of \eqref{class-BS}, $\text{Vol}(\wt K)=2\text{Vol}(K)$ and \eqref{eq-volume-polar-body}, we obtain 
\begin{eqnarray}\label{zs upper}
    \int_{\SS^{n}}\frac{1}{h_{ \widetilde z_{s}}^{n+1}(\zeta)}d\sigma\leq \frac{ (n+1){\rm{Vol}}(\mathbb{B}^{n+1})^{2}}{2{ {\rm{Vol}}(K)}}.
\end{eqnarray}
Since $\ell(\zeta)=1-\cos\theta\<\zeta, E_{n+1}\><1$ on $\SS^{n}_{\theta}$, then by applying the Jensen's inequality as \eqref{Jen-ine}, it yields
\begin{eqnarray}\label{upper-zs}
    \exp\left(\frac{1}{\omega_{\theta}}\int_{\SS_{\theta}^{n}}  \log \left(\frac{1}{u_{ \widetilde z_{s}}^{n+1}(\zeta)}\right) \ell(\zeta)d\sigma  \right) &\leq& \frac{1}{\omega_{\theta}}\int_{\SS_{\theta}^{n}}\frac{1}{u_{ \widetilde z_{s}}^{n+1}(\zeta)}\ell(\zeta)d\s \notag \\
    &\leq& \frac{1}{\omega_{\theta}}\int_{\SS_{\theta}^{n}}\frac{1}{h_{\widetilde z_{s}}^{n+1}(\zeta)}d\sigma.
\end{eqnarray}
Inserting \eqref{zs upper} into \eqref{upper-zs}, we obtain 
$$\frac{1}{\omega_{\theta}} \int_{\SS^{n}_{\theta}}\left(\log u_{\widetilde z_{s}}(\zeta)\right)\ell(\zeta) d\sigma\geq -C,$$
 together with \eqref{eq-entropy-on-K} and \eqref{eq-santalo-wilde-K-interior},  it follows  that $$\mathcal{F}_{\theta}(K)\geq -C,$$ thus we complete the proof of \eqref{ineq-two-sides-est-F-theta}.
 Notice that $K_{f}$ is a bounded closed convex body in $\RR^{n}=\partial{\ov{\RR^{n+1}_{+}}}$, following the same argument as in the proof of Propositions \ref{prop-exist-cap-entropy-point} and \ref{lem entropy}, we can show there exists a unique interior point $z_{e}\in \text{int}( K_f)$, such that $\mathcal{F}_{\theta}(K)=\frac{1}{\omega_{\theta}}\int_{\SS^{n}_{\theta}}\left(\log u_{z_{e}}(\zeta)\right)\ell(\zeta)d\sigma$. This completes the proof.
\end{proof}

Without causing confusion, we continue to call $z_{e}\coloneqq z_{e}(K)$ in Proposition \ref{prop-small angle} the \textit{capillary entropy point} of the $\theta$-capillary convex body $K\in \mathcal{G}_{\theta}$ associated with $\F_\theta(K)$. We next show that, under Hausdorff convergence, the limiting convex body of a sequence of capillary convex bodies in $\K_\theta$ is contained in $\G_\theta$. That is, its contact angle within  $\p\R^{n+1}_+$ is non-increasing.

\begin{prop}\label{prop-angle-decrease}
   Let $\{\widehat{\S}_{k}\}_{k\geq 1}\subset \mathcal{K}_{\theta}$ be a sequence of bounded capillary convex body, and $\widehat{\S}_{\star}\in \mathcal{G}$.  Suppose that  $\lim\limits_{k\rightarrow +\infty}\widehat{\S}_{k}=\widehat{\S}_{\star}$ in the  Hausdorff distance sense,  then $\widehat{\S}_{\star}\in \mathcal G_{\theta}$.

\end{prop}
\begin{proof}
 We just need to show that \begin{eqnarray}\label{angle-decrease}
       \vartheta(x)\leq \theta, \quad \text{for all }x\in \partial (\widehat\S_{\star}\cap \p{ \RR^{n+1}_{+}}).
   \end{eqnarray}
By assumption and \cite[Lemma 1.8.14]{Sch}, we have
\begin{eqnarray}\label{sup-func}
    h_{\widehat\S _{k}}(\zeta)\coloneqq \sup\limits_{y\in \widehat{\S}_{k}}\<y, \zeta\>\rightarrow h_{\widehat{\S}_{\star}}(\zeta),  ~\forall~ \zeta\in \SS^{n}, \text{ as } k\to +\infty,
\end{eqnarray}
where $h_{\widehat{\S}_{k}}$ and $h_{\widehat{\S}_{\star}}$ are the classical support functions of $\widehat{\S}_{k}$ and 
$\widehat{\S}_{\star}$ with respect to the origin respectively.

Given $\zeta\in \mathbb{S}^{n}\setminus {\rm{int}}(\SS^{n}_{\theta})$, then for each $\wh\S_{k}\in\mathcal{K}_{\theta}$, there exists a point $x_{k}\in \partial\S_{k}$, such that $h_{\widehat{\S}_{k}}(\zeta)=\<x_{k}, \zeta\>$. Let $k\rightarrow +\infty$, we have $x_{k}\rightarrow x$ (possibly up to a subsequence) for some $x\in \partial (\widehat{\S}_{\star}\cap \p\ov{\RR^{n+1}_{+}})$. Together with\eqref{sup-func} implies 
$  \zeta \in \mathcal{N}_{x}\widehat{\S}_{\star}$. Hence we obtain 
\begin{eqnarray*}
    \mathbb{S}^{n}\setminus {\rm{int}}(\SS^{n}_\theta)\subset \bigcup_{x\in\partial (\widehat\S_{\star}\cap \p { \RR^{n+1}_{+}}) }\mathcal{N}_{x}\widehat{\S}_{\star}.
\end{eqnarray*}
By  Definition \ref{defn-angle}, we conclude that \eqref{angle-decrease} holds.
\end{proof}

We conclude the appendix by showing that for all capillary convex body $\widehat{\S}\in \mathcal{K}_{\theta}^{\circ}$, under natural constraints on ${\rm{Vol}(\widehat{\S})}$ and $\mathcal{E}_{\theta}(\widehat{\S})$, the capillary entropy point $z_e(\wh\S)$ of $\widehat{\S}$ will stay away from $\S$. This result will play a pivotal role in deriving the positive lower bound of the solution to flow \eqref{eq-scalar-capillary-support-normalized-GCF}.
We adapt a compactness argument as in \cite[Lemma 4.2]{GN} and \cite[Lemma 4.4]{AGN} for the general non-smooth setting. However, under Hausdorff convergence, the limiting convex body may not be a capillary convex body, i.e., the constant angle condition may not be preserved.  To address this difficulty, we will leverage the concept of $\theta$-capillary convex body introduced in this section and its associated properties. 

We first introduce a set of capillary convex bodies.
For a positive constant $A>0$, we denote  
\begin{eqnarray*}
    \Gamma_{A}\coloneqq\left\{\widehat\S\in \mathcal{K}_{\theta}^{\circ}\mid  
     ~{\rm{Vol}}(\widehat{\S})={\rm{Vol}}(\widehat{\C_{\theta}}), ~\mathcal{E}_{\theta}(\wh\S)\leq A 
    \right \}.
\end{eqnarray*}
Now we are ready to state the result.

\begin{thm}\label{thm-lim-Hausdorff-distance}
    Let $\theta\in (0, \pi/ 2)$. We have
    \begin{enumerate}
        \item Suppose that $\{\widehat\S_{k}\}_{k\geq 1}\subset  \Gamma_{A}$, and 
    $\lim\limits_{k\rightarrow +\infty}\widehat\S_{k}=\widehat\S_\star$  in Hausdorff distance sense. Then $\widehat{\S}_{\star}\in \mathcal{G}_{\theta}$ and 
    \begin{eqnarray}\label{lim}
        \lim\limits_{k\rightarrow +\infty}\mathcal{E}_{\theta}(\widehat{\Sigma}_{k})=\mathcal F_{\theta}(\wh{\S}_{\star}).
    \end{eqnarray} 
    \item  Then there exists a constant $\delta > 0$, depending only on $n$ and $A$, such that for all $\widehat{\S}\subset \Gamma_{A}$, the capillary entropy point $z_e(\widehat{\S})$ of $\widehat{\S}$ satisfies 
    \begin{eqnarray}\label{dist}
{\rm{dist}}\left(z_e(\widehat{\S}), \S\right) \geq \delta.
\end{eqnarray}
    
    \end{enumerate}
\end{thm}
 
\begin{proof}

 \textit{(1)}
Since the volume is preserved along the Hausdorff convergence, this implies that $\widehat{\S}_{\star}$ cannot collapse into a lower-dimensional convex set, so we obtain $\widehat{\S}\in \mathcal{G}$. From Proposition \ref{prop-angle-decrease} above, we have $\widehat{\S}_{\star}\in \mathcal{G}_{\theta}$. Next, we proceed to prove \eqref{lim}.

From Proposition \ref{radius} and \eqref{double-control}, given $\widehat \S\in \Gamma_{A}$, there exists a positive constant $C_0$, depending only on $n$ and $A$, such that
\begin{eqnarray}\label{u-upper}
    u_z^{\S}\leq 2\rho_+(\wh\S,\theta) \leq 2 C e^{\E_\theta(\wh\S)} \leq 2Ce^A\eqqcolon C_0,
\end{eqnarray}
where $u_z^{\Sigma}$ is the capillary support function of $\widehat\S$ with respect to the point $z\in  \widehat{\partial \S}$.

 By
 Proposition \ref{prop-small angle}, we derive that the capillary entropy point $ z_e(\wh{\S}_\star)$ of $\widehat\S_\star$ lies in ${\rm{int}}(\widehat{\partial\S}_\star)$ with $\widehat{\partial \S}_{\star}\coloneqq \widehat{\S}_{\star}\cap \partial \ov{\RR^{n+1}_{+}}$, therefore, for $k$ sufficiently large, we have $ z_e({\wh\S_\star})\in {\rm{int}}(\widehat{\partial\S}_k)$. In view of \eqref{u-upper} and the dominated convergence theorem, we derive
\begin{eqnarray}\label{inf}
  \mathcal F_{\theta}(\wh\S_\star)&=&\frac{1}{\omega_{\theta}}\int_{\SS^{n}_{\theta}}\left(\log u_{z_e(\wh\S_\star)}^{\S_\star}(\zeta)\right)\ell (\zeta)d\s\notag \\
    &=&\lim\limits_{k\rightarrow\infty}\frac{1}{\omega_{\theta}}\int_{\C_{\theta}}\left(\log u_{  z_e(\wh\S_\star)}^{\S_{k}}(\xi)\right)\ell(\xi) d\s\leq \lim\limits_{k\rightarrow +\infty}\mathcal{E}_{\theta}(
    \wh\S_k ),~~~
\end{eqnarray}where $u_{ z_e(\wh\S_\star)}^{\Sigma_k}$ denote the capillary support function of $\widehat{\S}_k$ with respect to the point $ z_e(\wh{\S}_\star)$ and $u_{  z_e(\wh\S_\star)}^{\Sigma_\star}=\frac{h^{ \S_{\star}}_{{z}_{e}(\widehat\S_{\star})}}{\ell}$, where $h^{ \S_{\star}}_{{z}_{e}(\widehat\S_{\star})}$ is the classical support function of $\widehat{\S}_{\star}$ with respect to the point ${z}_{e}(\widehat{\S}_{\star})$.
On the other hand,  Proposition \ref{lem-low-entroy} implies 
\begin{eqnarray*}
 \frac{1}{\omega_{\theta}}\int_{\C_{\theta}}\log \left(\frac{u_{z_e(\wh\S_{k})}^{\S_{k}}(\xi)}{C_0}\right)\ell(\xi) d\s=\mathcal{E}_{\theta}(\wh\S_k )-\frac{1}{\omega_{\theta}}\int_{\C_{\theta}} (\log C_0) \ell(\xi) d\s\geq -C,
\end{eqnarray*}for some positive constant $C$, depending only on $n$ and $A$.
Together with \eqref{u-upper} again, it implies
\begin{eqnarray*}
  \frac{1}{\omega_{\theta}}\int_{\C_{\theta}}\left|\log \left(\frac{u^{\S_{k}}_{z_e(\wh\S_{k})}(\xi)}{C_0} \right)\ell(\xi) d\s \right|  d\s\leq C.
\end{eqnarray*}
Let $z_\star\coloneqq \lim\limits_{k\rightarrow +\infty}z_e({\wh\S_{k}})$. Since  the capillary entropy point $z_e({\wh\S_{k}})\in {\rm{int}}(\wh{\p\S}_k)$ and $\lim\limits_{k\rightarrow +\infty}\wh\S_k =\widehat\S_\star$, it follows that $z_\star\in \widehat{\partial\S}_\star$. By Fatou's Lemma, we obtain 
\begin{eqnarray}\label{upper-lim}
  \frac{1}{\omega_{\theta}}\int_{\SS^{n}_{\theta}}\log  \left(\frac{u^{\S_\star}_{z_\star}(\zeta)}{C_0}\right)\ell(\zeta) d\s \geq \limsup\limits_{k\rightarrow +\infty}\frac{1}{\omega_{\theta}}\int_{\C_{\theta}}\log \left(\frac{u_{z_e(\wh\S_{k})}^{\S_{k}}(\xi)}{C_0}\right)\ell(\xi) d\s,
\end{eqnarray}
which implies 
\begin{eqnarray}\label{sup}
    \mathcal{F}_{\theta}(\widehat{\S}_{\ast})\geq \limsup\limits_{k\rightarrow +\infty}\mathcal{E}_{\theta}(\wh\S_k ).
\end{eqnarray}
In view of \eqref{inf} and \eqref{sup}, we conclude that the assertion \eqref{lim} holds.

 \textit{(2)} Now we move on to prove \eqref{dist} by arguing via contradiction. Suppose there exists a sequence $\{\wh\S_{l}\}_{l\geq 1}\subset \Gamma_{A}$ such that
\begin{eqnarray}
    {\rm{dist}}\left(z_e(\wh\S_{l}), \S_{l}\right)\rightarrow 0,~{\rm{as}}~l\rightarrow \infty. \label{dist-1}
\end{eqnarray}
By the Blaschke selection theorem, see, e.g., \cite[Theorem 1.8.7]{Sch},
there exists a subsequence of $\{\wh\S_{l}\}_{l\geq 1}$, still denoted by $\{\wh\S_{l}\}_{l\geq 1}$, such that $\wh\S_l$ converges to $\wh\S_\star$ as $l\to +\infty$. Let $q_{0}\coloneqq\lim\limits_{l\rightarrow +\infty}z_e(\wh\S_{l})$. Since each $z_e(\wh\Sigma_{l})\in\widehat{\partial\S} _{l}\subset \p{\RR^{n+1}_{+}}$,  together with  \eqref{dist-1}, it yields that  \begin{eqnarray}\label{fact-q0-boundary}
    q_{0}\in \partial \S_\star.
\end{eqnarray}  Furthermore, from \eqref{lim} and \eqref{upper-lim}, we have
\begin{eqnarray*}
    \mathcal{F}_{\theta}(\widehat{\S}_{\star}) &=&\lim\limits_{l\rightarrow +\infty}\mathcal{E}_{\theta}(\wh\S_l )=\lim\limits_{l\rightarrow +\infty}\frac{1}{\omega_{\theta}}\int_{\C_{\theta}}\left(\log u^{\S_{l}}_{z_e(\wh\S_{l})}(\xi)\right)\ell(\xi) d\s \\&\leq & \frac{1}{\omega_{\theta}}\int_{\SS^{n}_{\theta}}\left(\log u_{q_{0}}^{\S_\star}(\zeta)\right)\ell(\zeta) d\s\leq \mathcal{F}_{\theta}(\widehat{\S}_{\star}),
\end{eqnarray*}
 which implies that $q_0$ is the  capillary entropy point of $\widehat{\S}_{\star}$. By applying Proposition \ref{prop-small angle}, we know that the point $q_0 \in \text{int}(\widehat{\partial \S}_\star)$. However, this contradicts \eqref{fact-q0-boundary}. Hence, we complete the proof.
\end{proof}

\bigskip

\bigskip

\noindent\textit{Acknowledgment:} X.M. was supported by the National Key R $\&$ D Program of China (No. 2020YFA0712800) and the Postdoctoral Fellowship Program of CPSF under Grant Number GZC20240052.  L.W. was partially supported by CRM De Giorgi of Scuola Normale Superiore and PRIN Project 2022E9CF89 of the University of Pisa. 

\bigskip

\printbibliography

\end{document}